\documentclass[11pt]{amsart}
\usepackage{amssymb,amsmath,amsfonts,amsthm,mathrsfs,fullpage}
\usepackage[all]{xy} \SelectTips{cm}{10}
\usepackage{color}
\numberwithin{equation}{section}

\renewcommand{\AA}{\mathbb A}

\newcommand{\CC}{\mathbb C}

\newcommand{\FF}{\mathbb F}
\newcommand{\GG}{\mathbb G}

\newcommand{\QQ}{\mathbb Q}
\newcommand{\RR}{\mathbb R}

\newcommand{\ZZ}{\mathbb Z}

\newcommand{\calA}{\mathcal A}

\newcommand{\calC}{\mathcal C}

\newcommand{\calG}{\mathcal G}

\newcommand{\calT}{\mathcal T}
\newcommand{\calS}{\mathcal S}
\newcommand{\calW}{\mathcal W}

\newcommand{\calV}{\mathcal V}

\newcommand{\scrS}{\mathscr S}

\def\T{\mathbf{T}}  
\def\G{\mathbf{G}}  

\def\cl{\operatorname{cl}} 
\def\ad{{\operatorname{ad}}}
\def\sc{{\operatorname{sc}}}
\def\Conj{\operatorname{Conj}}

\def\Spec{\operatorname{Spec}} 

 \def\Gal{\operatorname{Gal}}
\def\End{\operatorname{End}}

\def\cl{\operatorname{cl}}

\def\conn{{\operatorname{conn}}}
\def\der{{\operatorname{der}}}

\def\Spec{\operatorname{Spec}}

\def\tr{\operatorname{tr}}

\def\Gal{\operatorname{Gal}}

\def \GL{\mathbf{GL}} 
\def \SL {\mathbf{SL}} 
\def \Sp {\mathbf{Sp}} 
\def \SO {\mathbf{SO}}

\def \GSp {\mathbf{GSp}} 

\def\Aut{\operatorname{Aut}} 
 
\def\End{\operatorname{End}}
\def\Frob{\operatorname{Frob}}

\def\Res{\operatorname{Res}}

\def\tr{\operatorname{tr}}

 \def \Aut {\operatorname{Aut}}

\definecolor{purple}{rgb}{1,0,1}

\def\bbar#1{\setbox0=\hbox{$#1$}\dimen0=.2\ht0 \kern\dimen0 
\overline{\kern-\dimen0 #1}}
\newcommand{\Qbar}{{\overline{\mathbb Q}}} 
\newcommand{\Kbar}{\bbar{K}}

\newcommand{\kbar}{\bbar{k}} 
\newcommand{\FFbar}{\overline{\FF}}

\newcommand{\defi}[1]{\textsf{#1}}  
\newtheorem{thm}{Theorem}[section]
\newtheorem{lemma}[thm]{Lemma}
\newtheorem{cor}[thm]{Corollary}
\newtheorem{prop}[thm]{Proposition}

\theoremstyle{definition}

\newtheorem{definition}[thm]{Definition}

\newtheorem{conj}[thm]{Conjecture}

\theoremstyle{remark}
\newtheorem{remark}[thm]{Remark}

\newenvironment{romanenum}{\hfill \begin{enumerate} }{\end{enumerate}}
\newenvironment{alphenum}{\hfill \begin{enumerate} }{\end{enumerate}}

\definecolor{webcolor}{rgb}{0.8,0,0.2}
\definecolor{webbrown}{rgb}{.6,0,0}
\usepackage[
        colorlinks,
        linkcolor=webbrown,  filecolor=webcolor,  citecolor=webbrown, 
        backref,
        pdfauthor={David Zywina}, 
        pdftitle={The splitting of reductions of an abelian variety},
]{hyperref}
\usepackage[alphabetic,backrefs,lite]{amsrefs} 

\begin{document}
\title{The splitting of reductions of an abelian variety}
\subjclass[2000]{Primary 14K15; Secondary 11F80}

\author{David Zywina}
\address{Department of Mathematics and Statistics, Queen's University, Kingston, ON  K7L~3N6, Canada}
\email{zywina@mast.queensu.ca}
\urladdr{http://www.mast.queensu.ca/\~{}zywina}
\date{\today}

\maketitle

\begin{abstract} 
Consider an absolutely simple abelian variety $A$ defined over a number field $K$.   For most places $v$ of $K$, we study how the reduction $A_v$ of $A$ modulo $v$ splits up to isogeny.   Assuming the Mumford-Tate conjecture for $A$ and possibly increasing $K$, we will show that $A_v$ is isogenous to the $m$-th power of an absolutely simple abelian variety for all places $v$ of $K$ away from a set of density $0$, where $m$ is an integer depending only on the endomorphism ring $\End(A_{\Kbar})$.   This proves many cases, and supplies justification, for a conjecture of Murty and Patankar.    Under the same assumptions, we will also describe the Galois extension of $\QQ$ generated by the Weil numbers of $A_v$ for most $v$.\\
\end{abstract}

\section{Introduction}
Consider a non-zero abelian variety $A$ defined over a number field $K$.  Let $\Sigma_K$ be the set of finite places of $K$, and for each place $v\in\Sigma_K$ let $\FF_v$ be the corresponding residue field.   The abelian variety $A$ has good reduction at all but finitely many places $v\in\Sigma_K$.   For a place $v\in \Sigma_K$ for which $A$ has good reduction, the reduction $A$ modulo $v$ is an abelian variety $A_v$ defined over $\FF_v$.   We know that the abelian variety $A_v$ is isogenous to a  product of simple abelian varieties (all defined over $\FF_v$).     The goal of this paper is study how $A_v$ factors for ``almost all'' places $v$, i.e., for  those $v$ away from a subset of $\Sigma_K$ with (natural) density 0.  In particular, we will supply evidence for the following conjecture of V.~K.~Murty and V.~Patankar \cite{MR2426750}.

\begin{conj}[Murty-Patankar] \label{C:MP}
Let $A$ be an absolutely simple abelian variety over a number field $K$.  
Let $\calV$ be the set of finite places $v$ of $K$ for which $A$ has good reduction and  $A_v/\FF_v$ is simple.   Then, after possibly replacing $K$ by a finite extension, the density of $\calV$ exists and $\calV$ has density 1 if and only if $\End(A_{\Kbar})$ is commutative.
\end{conj}

The conjecture in \cite{MR2426750} is stated without the condition that $K$ possibly needs to be replaced by a finite extension.   An extra condition is required since one can find counterexamples to the original conjecture.   (For example, let $A/\QQ$ be the Jacobian of the smooth projective curve defined by the equation $y^2=x^5-1$.   We have $\End(A_{\Qbar})\otimes_{\ZZ} \QQ=\QQ(\zeta_5)$, and in particular $A$ is absolutely simple.    For each prime $p\equiv -1\pmod{5}$, the abelian variety $A$ has good reduction at $p$ and $A_p$ is isogenous to $E_p^2$ where $E_p$ is an elliptic curve over $\FF_p$ that satisfies $|E_p(\FF_p)|=p+1$.  
Conjecture~1.1 will hold for $A$ with $K=\QQ(\zeta_5)$, equivalently, $A_p$ is simple for all primes $p\equiv 1 \pmod{5}$ away from a set of density 0.)

In our work, it will be necessary to first replace $K$ by a certain finite Galois extension $K_A^\conn$.   The field $K_A^\conn$ is the smallest extension of $K$ for which all the $\ell$-adic monodromy groups associated to $A$ over $K^{\conn}_A$ are connected, cf.~\S\ref{SS:monodromy groups}.   We quickly give an alternative description of this field from \cite{MR1339927}. For each prime number $\ell$, let $A[\ell^\infty]$ be the subgroup of $A(\Kbar)$ consisting of those points whose order is some power of $\ell$.   Let $K(A[\ell^\infty])$ be the smallest extension of $K$ in $\Kbar$ over which all the points of $A[\ell^\infty]$ are defined.   We then have 
\begin{equation*} \label{E:Kconn def LP}
K_A^\conn=\bigcap_\ell K(A[\ell^\infty]).
\end{equation*}

We shall relate Conjecture~\ref{C:MP} to the arithmetic of the Mumford-Tate group of $A$, see \S\ref{SS:MT group} for a definition of this group and \S\ref{SS:MT conj} for a statement of the Mumford-Tate conjecture for $A$.  

\begin{thm} \label{T:main}
Let $A$ be an absolutely simple abelian variety defined over a number field $K$ such that $K^{\conn}_A=K$.   Define the integer $m=[\End(A)\otimes_\ZZ\QQ:E]^{1/2}$ where $E$ is the center of the division algebra $\End(A)\otimes_\ZZ \QQ$.
\begin{romanenum}
\item
For all $v\in \Sigma_K$ away from a set of density 0,  $A_v$ is isogenous to $B^m$ for some abelian variety $B/\FF_v$.
\item
Suppose that the Mumford-Tate conjecture for $A$ holds.  Then for all $v\in \Sigma_K$ away from a set of density 0, $A_v$ is isogenous to $B^m$ for some absolutely simple abelian variety $B/\FF_v$.
\end{romanenum}
\end{thm}

\begin{cor} \label{C:main}
Let $A$ be an absolutely simple abelian variety defined over a number field $K$ such that $K^{\conn}_A=K$.    Let $\calV$ be the set of finite places $v$ of $K$ for which $A$ has good reduction and $A_v/\FF_v$ is simple.  If $\End(A)$ is non-commutative, then $\calV$ has density 0.  If $\End(A)$ is commutative and the Mumford-Tate conjecture for $A$ holds, then $\calV$ has density 1.  
\end{cor}

We will observe later that $\End(A_{\Kbar})\otimes_\ZZ \QQ=\End(A)\otimes_\ZZ \QQ$ when $K_A^{\conn}=K$; so the above corollary shows that Conjecture~\ref{C:MP} is a consequence of the Mumford-Tate conjecture.  Using Theorem~\ref{T:main}, we will prove the following general version. 

%

\begin{thm} \label{T:general}
Let $A$ be a non-zero abelian variety defined over a number field $K$ such that $K_A^{\conn}=K$.  The abelian variety $A$ is isogenous to $A_1^{n_1}\times\cdots \times A_s^{n_s}$ where the abelian varieties $A_i/K$ are simple and pairwise non-isogenous.   For $1\leq i\leq s$, define the integer $m_i=[\End(A_i)\otimes_\ZZ \QQ:E_i]^{1/2}$ where $E_i$ is the center of $\End(A_i)\otimes_\ZZ \QQ$.

Suppose that the Mumford-Tate conjecture for $A$ holds.    Then for all places $v\in \Sigma_K$ away from a set of density 0, $A_v$ is isogenous to a product $\prod_{i=1}^s B_i^{m_i n_i}$ where the $B_i$ are absolutely simple abelian varieties over $\FF_v$ which are pairwise non-isogenous and satisfy $\dim(B_i)=\dim(A_i)/m_i$.
\end{thm}

Observe that the integer $s$ and the pairs $(n_i,m_i)$ from Theorem~\ref{T:general} can be determined from the endomorphism ring $\End(A)\otimes_\ZZ\QQ$.

\subsection{The Galois group of characteristic polynomials}
Let $A$ be a non-zero abelian variety over a number field $K$.   Fix a finite place $v$ of $K$ for which $A$ has good reduction.   Let $\pi_{A_v}$ be the Frobenius endomorphism of $A_v$ and let $P_{A_v}(x)$ be the characteristic polynomial of $\pi_{A_v}$.  The polynomial $P_{A_v}(x)$ is monic of degree $2\dim A$ with integral coefficients and can be characterized by the property that $P_{A_v}(n)$ is the degree of the isogeny $n-\pi_{A_v}$ of $A$ for each integer $n$.  

Let $\calW_{A_v}$ be the set of roots of $P_{A_v}(x)$ in $\Qbar$.  Honda-Tate theory says that $A_v$ is isogenous to a power of a simple abelian variety if and only if $P_{A_v}(x)$ is a power of an irreducible polynomial; equivalently, if and only if the action of $\Gal_\QQ$ on $\calW_{A_v}$ is transitive.   \\

 The following theorem will be important in the proof of Theorem~\ref{T:main}.    Let $\G_A$ be the Mumford-Tate group of $A$; it is a reductive group over $\QQ$ which we will recall in \S\ref{SS:MT group}.  Let $W(\G_A)$ be the Weyl group of $\G_A$.   We define the \defi{splitting field} $k_{\G_A}$ of $\G_A$ to be the intersection of all the subfields $L\subseteq \Qbar$ for which the group $\G_{A,L}$ is split.

\begin{thm} \label{T:Weil under MT}  
Let $A$ be an absolutely simple abelian variety over a number field $K$ that satisfies $K_A^{\conn}=K$.   Assume that the Mumford-Tate conjecture for $A$ holds and let $L$ be a finite extension of $k_{\G_A}$.  Then 
\[
\Gal(L(\calW_{A_v})/L)\cong W(\G_A).
\]
for all places $v\in \Sigma_K$ away from a set of density 0.
\end{thm}

Moreover, we expect the following:

\begin{conj} \label{C:main}
Let $A$ be a non-zero abelian variety over a number field $K$ that satisfies $K_A^{\conn}=K$.   There is a group $\Pi(\G_A)$ such that $\Gal(\QQ(\calW_{A_v})/\QQ)\cong \Pi(\G_A)$ for all $v\in \Sigma_K$ away from a set with natural density 0.
\end{conj}

We shall later give an explicit candidate for the group $\Pi(\G_A)$; it has $W(\G_A)$ as a normal subgroup.  We will also prove the conjecture in several cases, cf.~\S\ref{S:conj ends}.

\subsection{Some previous results}
We briefly recall a few earlier known cases of Theorems~\ref{T:main} and \ref{T:Weil under MT}.

Let $A$ be an abelian variety over a number field $K$ such that $\End(A_{\Kbar})=\ZZ$ and such that $2\dim(A)$ is not a $k$-th power and not of the form $2k \choose k$ for every odd $k>1$.  Under these assumptions, Pink has shown that $\G_A$ is isomorphic to $\GSp_{2\dim(A),\QQ}$ and that the Mumford-Tate conjecture for $A$ holds \cite{MR1603865}*{Theorem~5.14}.   We will have $K_A^\conn=K$, so Theorem~\ref{T:main} says that $A_v/\FF_v$ is absolutely simple for all places $v\in \Sigma_K$ away from a set of density 0.  We have $k_{\G_A}=\QQ$ since $\G_A$ is split, so Theorem~\ref{T:Weil under MT} implies that $\Gal(\QQ(\calW_{A_v})/\QQ)$ is isomorphic to the Weyl group $W(\GSp_{2\dim(A),\QQ})=W(\Sp_{2\dim(A),\QQ})\cong W(C_{\dim(A)})$ for all $v\in\Sigma_K$ away from a set of density 0.   These results are due to Chavdarov \cite{MR1440067}*{Cor.~6.9} in the special case where $\dim(A)$ is 2, 6 or odd (these dimensions are used to cite a theorem of Serre which gives a mod $\ell$ version of Mumford-Tate).\\

Now consider the case where $A$ is an absolutely simple CM abelian variety defined over a number field $K$; so $F:=\End(A_{\Kbar})\otimes_\ZZ\QQ$ is a number field that satisfies $[F:\QQ]=2\dim(A)$.   After replacing $K$ by a finite extension, we may assume that $F=\End(A)\otimes_\ZZ \QQ$.   We have $\G_A\cong \Res_{F/\QQ}(\GG_{m,F})$, where $\Res_{F/\QQ}$ denotes restriction of scalars from $F$ to $\QQ$.  The theory of complex multiplication shows that $A$ satisfies the Mumford-Tate conjecture and hence Theorem~\ref{T:main} says that $A_v/\FF_v$ is absolutely simple for almost all places $v\in \Sigma_K$; this is also \cite{MR2426750}*{Theorem~3.1} where it is proved using $L$-functions and Hecke characters.   Theorem~\ref{T:Weil under MT} is not so interesting in this case since $W(\G_A)=1$.\\

Several cases of Theorem~\ref{T:main} were proved by J.~Achter in \cite{MR2496739} and \cite{Achter-effective}; for example, those abelian varieties $A/K$ such that $F:=\End(A_{\Kbar})\otimes_\ZZ \QQ$ is a totally real number field and $\dim(A)/[F:\QQ]$ is odd.   A key ingredient is known cases of the Mumford-Tate conjecture from the papers \cite{MR2400251}, \cite{MR2290584} and \cite{MR2663452}.  Achter's approach is very similar to this paper and boils down to showing that $P_{A_v}(x)$ is an appropriate power of an irreducible polynomial for almost all places $v\in \Sigma_K$. In the case where $\End(A_{\Kbar})$ is commutative, he uses the basic property that if $P_{A_v}(x) \bmod{\ell}$ is irreducible in $\FF_\ell[x]$, then $P_{A_v}(x)$ is irreducible in $\ZZ[x]$; unfortunately, this approach will not work for all absolutely simple abelian varieties $A/K$ for which $K_A^\conn=K$ and $\End(A_{\Kbar})$ is commutative (see \S\ref{SS:Mumford type} for an example).    Corollary~\ref{C:main} in the non-commutative case also follows from \cite{MR2496739}*{Theorem~B}.

\subsection{Overview}
We set some notation.  The symbol $\ell$ will always denote a rational prime.  If $X$ is a scheme over a ring $R$ and we have a ring homomorphism $R\to R'$, then we denote by $X_{R'}$ the scheme $X\times_{\Spec R} \Spec R'$ over $R'$.   The homomorphism is implicit in the notation; it will usually be a natural inclusion or quotient homomorphism; for example, $\QQ\to \QQ_\ell$, $\ZZ_\ell\to \QQ_\ell$, $\ZZ_\ell\to \FF_\ell$, $K\hookrightarrow \Kbar$.  For a field $K$, we will denote by $\Kbar$ a fixed algebraic closure and define the absolute Galois group $\Gal_K:=\Gal(\Kbar/K)$.\\

Let $A$ be a non-zero abelian variety over a number field $K$ that satisfies $K^{\conn}_A=K$.  Fix an embedding $K\subseteq \CC$.    In \S\ref{S:AV background}, we review the basics about the $\ell$-adic representations arising from the action of $\Gal_K$ on the $\ell$-power torsion points of $A$.  To each prime $\ell$, we will  associate an algebraic group $\G_{A,\ell}$ over $\QQ_\ell$.   Conjecturally, the connected components of the groups $\G_{A,\ell}$ are isomorphic to the base extension of a certain reductive group $\G_A$ defined over $\QQ$; this is the Mumford-Tate group of $A$ and comes with a faithful action on $H_1(A(\CC),\QQ)$.   In \S\ref{S:reductive}, we review some facts about reductive groups and in particular define the group $\Pi(\G_A)$ of Conjecture~\ref{C:main}.\\

Let us hint at how Theorem~\ref{T:main} and Theorem~\ref{T:Weil under MT} are connected; further details will be supplied later.  For the sake of simplicity, suppose that $\End(A_{\Kbar})=\ZZ$.    Fix a maximal torus $\T$ of $\G_A$ and a number field $L$ for which $\T_L$ is split.  Let $X(\T)$ be the group of characters of $\T_{\Qbar}$ and let $\Omega\subseteq X(\T)$ be the set of weights arising from the representation of $\T\subseteq \G_A$ on $H_1(A(\CC),\QQ)$.   The Weyl group $W(\G_A,\T)$ has a natural faithful action on the set $\Omega$.    

Using the geometry of $\G_A$ and our additional assumption $\End(A_{\Kbar})=\ZZ$, one can show that action of $W(\G_A,\T)$ on $\Omega$ is transitive.   Assuming the Mumford-Tate conjecture, we will show that for all $v\in \Sigma_K$ away from a set of density 0 there is a bijection $\calW_{A_v} \leftrightarrow \Omega$ such that the action of $\Gal_L$ on $\calW_{A_v}$ corresponds with the action of some subgroup of $W(\G_A,\T)$ on $\Omega$ (this will be described in \S\ref{SS:Galois action} and it makes vital use of a theorem of Noot described in \S\ref{S:Noot}).  So for almost all $v\in\Sigma_K$, we find that $\Gal(L(\calW_{A_v})/L)$ is isomorphic to a subgroup of $W(\G_A,\T)$.  If $\Gal(L(\calW_{A_v})/L)$ is isomorphic to $W(\G_A,\T)$, then we deduce that $\Gal_L$ acts transitively on $\calW_{A_v}$ and hence $P_{A_v}(x)$ is a power of an irreducible polynomial.   The assumption $\End(A_{\Kbar})=\ZZ$ also ensures that $P_{A_v}(x)$ is separable for almost all $v$, and thus we deduce that $P_{A_v}(x)$ is almost always irreducible (and hence $A_v$ is almost always simple).   To show that $\Gal(L(\calW_{A_v})/L)$ is maximal for all $v\in \Sigma_K$ away from a set of density 0, we will use a version of Jordan's lemma with some local information from \S\ref{S:local Galois}.\\

The proof of Theorem~\ref{T:general} can be found in \S\ref{S:general proof}; it easily reduced to the absolutely simple case.  In \S\ref{S:conj ends}  we discuss Conjecture~\ref{C:main} further and give an extended example. Finally, we will prove effective versions of Theorem~\ref{T:main} and Theorem~\ref{T:Weil under MT} in \S\ref{S:large sieve}.

\subsection*{Acknowledgements} 

Thanks to J.~Achter for rekindling the author's interest in the conjecture of Murty and Patankar.   Thanks to F.~Jouve and E.~Kowalski; many of the techniques and strategies used here were first worked out in the joint paper \cite{JKV-splitting}.

\section{Abelian varieties and Galois representations: background} \label{S:AV background}

Fix a non-zero abelian variety $A$ defined over a number field $K$.   In this section, we review some theory concerning the $\ell$-adic representations associated to $A$.  In particular, we will define the Mumford-Tate group of $A$ and state the Mumford-Tate conjecture.  For basics on abelian varieties see \cite{MR861974}.   The papers \cite{MR0476753} and \cite{MR1265537} supply overviews of several motivic conjectures for $A$ and how they conjecturally relate with its $\ell$-adic representations.

\subsection{Characteristic polynomials} \label{SS:Honda-Tate}

Fix a finite field $\FF_q$ with cardinality $q$.   Let $B$ be a non-zero abelian variety defined over $\FF_q$ and let $\pi_B$ be the Frobenius endomorphism of $B$.     The \defi{characteristic polynomial} of $B$ is the unique polynomial $P_{B}(x)\in \ZZ[x]$ for which the isogeny $n-\pi_{B}$ of $B$ has degree $P_{B}(n)$ for all integers $n$.  The polynomial $P_{B}(x)$ is monic of degree $2\dim B$ and its roots in $\CC$ have absolute value $q^{1/2}$.    

The following lemma says that, under some additional conditions, the factorization of $P_{B}(x)$ into irreducible polynomials is identical to the factorization of $B$ into simple abelian varieties.

\begin{lemma} \label{L:Honda-Tate}
Let $B$ be a non-zero abelian variety defined over $\FF_p$ where $p$ is a prime.   Assume that $P_{B}(x)$ is not divisible by $x^2-p$.   If $P_{B}(x)=\prod_{i=1}^s Q_i(x)^{m_i}$ where the $Q_i(x)$ are distinct monic irreducible polynomials in $\ZZ[x]$, then $B$ is isogenous to $\prod_{i=1}^s B_i^{m_i}$ where $B_i$ is a simple abelian variety over $\FF_p$ satisfying $P_{B_i}(x)=Q_i(x)$.
\end{lemma}
\begin{proof}   This is a basic application of Honda-Tate theory; see \cite{MR0314847}.  We know that $B$ is isogenous to $\prod_{i=1}^s B_i^{n_i}$ where the $B_i/\FF_p$ are simple and pairwise non-isogenous.  We have $P_{B}(x)=\prod_{i=1}^s P_{B_i}(x)^{n_i}$.  Honda-Tate theory says that $P_{B_i}(x)=Q_i(x)^{e_i}$ where the $Q_i(x)$ are distinct irreducible monic polynomials in $\ZZ[x]$ and the $e_i$ are positive integers.   After possibly reordering the $B_i$, we have the factorization of $P_{B}(x)$ in the statement of the lemma with $m_i=n_i e_i$.  It thus suffices to show that $e_i=1$ for $1\leq i \leq s$.

Fix $1\leq i \leq s$.  Since $B_i$ is simple, the ring $E:=\End(B_i)\otimes_\ZZ \QQ$ is a division algebra with center $\Phi:=\QQ(\pi_{B_i})$.   By \cite{MR0314847}*{I.~Theorem~8}, we have  $e_i=[E:\Phi]^{1/2}$.   If the number field $\Phi$ has a real place, then $\pi_{B_i} = \pm p^{1/2}$ since $|\pi_{B_i}|= p^{1/2}$.     However, if $\pi_{B_i}=\pm p^{1/2}$, then $x^2- p$ divides $P_{B_i}(x)\in \QQ[x]$.   So by our assumption that $x^2- p$ does not divide $P_{B}(x)$, we conclude that $\Phi$ has no real places.  Using that $\QQ(\pi_{B_i})$ has no real places and $\FF_p$ has prime cardinality, \cite{MR0265369}*{Theorem~6.1} implies that $E$ is commutative and hence $e_i=[E:\Phi]^{1/2}=1$. 
\end{proof}

We define $\calW_{B}$ to be the set of roots of $P_{B}(x)$ in $\Qbar$.    The elements of $\calW_{B}$ are algebraic integers with absolute value $q^{1/2}$ under any embedding $\Qbar\hookrightarrow \CC$.    We define $\Phi_{B}$ to be the subgroup of $\Qbar^\times$ generated by $\calW_B$.

\subsection{Galois representations}
  For each positive integer $m$, let $A[m]$ be the $m$-torsion subgroup of $A(\Kbar)$; it is a free $\ZZ/m\ZZ$-module of rank $2\dim A$.  For a fixed rational prime $\ell$, let $T_\ell(A)$ be the inverse limit of the groups $A[\ell^e]$ where the transition maps are multiplication by $\ell$.   We call $T_\ell(A)$ the \defi{Tate module} of $A$ at $\ell$; it is a free $\ZZ_\ell$-module of rank $2\dim A$.   Define $V_\ell(A)= T_\ell(A) \otimes_{\ZZ_\ell} \QQ_\ell$.   There is a natural action of $\Gal_K$ on the groups $A[m]$, $T_\ell(A)$, and $V_\ell(A)$.  Let
\[
\rho_{A,\ell} \colon \Gal_K \to \Aut_{\QQ_\ell}(V_\ell(A)).
\]
be the Galois representation which describes the Galois action on the $\QQ_\ell$-vector space $V_\ell(A)$.   

Fix a finite place $v$ of $K$ for which $A$ has good reduction.  Denote by $A_v$ the abelian variety over $\FF_v$ obtained by reducing $A$ modulo $v$.   If $v\nmid \ell$, then $\rho_{A,\ell}$ is unramified at $v$ and satisfies
\[
P_{A_v}(x)=\det(xI-\rho_{A,\ell}(\Frob_v))
\]
where $P_{A_v}(x)\in \ZZ[x]$ is the degree $2\dim A$ monic polynomial of \S\ref{SS:Honda-Tate}.   Furthermore, $\rho_{A,\ell}(\Frob_v)$ is a semisimple element of $\Aut_{\QQ_\ell}(V_\ell(A))\cong \GL_{2\dim A}(\QQ_\ell)$.

\subsection{$\ell$-adic monodromy groups} \label{SS:monodromy groups}

Let $\GL_{V_\ell(A)}$ be the algebraic group defined over $\QQ_\ell$ for which $\GL_{V_\ell(A)}(L)=\Aut_L(L \otimes_{\QQ_\ell} V_\ell(A))$ for all field extensions $L/\QQ_\ell$.  The image of $\rho_{A,\ell}$ lies in $\GL_{V_\ell(A)}(\QQ_\ell)$.   Let $\G_{A,\ell}$ be the Zariski closure in $\GL_{V_\ell(A)}$ of $\rho_{A,\ell}(\Gal_K)$; it is an algebraic subgroup of $\GL_{V_\ell(A)}$ called the \defi{$\ell$-adic algebraic monodromy group} of $A$.  Denote by $\G_{A,\ell}^\circ$ the identity component of $\G_{A,\ell}$.     

Let $K^\conn_A$ be the fixed field in $\Kbar$ of $\rho_{A,\ell}^{-1}(\G_{A,\ell}^\circ(\QQ_\ell))$; it is a finite Galois extension of $K$ that does not depend on $\ell$, cf.~\cite{MR1730973}*{133~p.17}.   Thus for any finite extension $L$ of $K^\conn_A$ in $\Kbar$, the group $\rho_{A,\ell}(\Gal(\Kbar/L))$ is Zariski dense in $\G_{A,\ell}^\circ$ (equivalently, $\G_{A_L,\ell}=\G_{A,\ell}^\circ$).   We have $K^\conn_A=K$ if and only if all the $\ell$-adic monodromy groups $\G_{A,\ell}$ are connected.

\begin{prop}  \label{P:Faltings}
Assume that $K_A^\conn=K$.
\begin{romanenum}
\item The commutant of $\G_{A,\ell}$ in $\End_{\QQ_\ell}(V_\ell(A))$ is naturally isomorphic to $\End(A)\otimes_\ZZ \QQ_\ell$.   
\item
The group $\G_{A,\ell}$ is reductive.
\item 
We have $\End(A_{\Kbar})\otimes_\ZZ \QQ = \End(A)\otimes_\ZZ \QQ$.
\end{romanenum}
\end{prop}
\begin{proof}
Faltings proved that the representation $\rho_{A,\ell}$ is semisimple and that the natural homomorphism
\[
\End(A)\otimes_\ZZ \QQ_\ell \to \End_{\QQ_\ell[\Gal_K]}(V_\ell(A))
\]
is an isomorphism, cf.~\cite{MR861971}*{Theorems 3--4}.   So the commutant of $\G_{A,\ell}$ in $\End_{\QQ_\ell}(V_\ell(A))$ equals $\End(A)\otimes_\ZZ \QQ_\ell$. It follows easily that $\G_{A,\ell}$ is reductive.

Let $L$ be a finite extension of $K$ for which $\End(A_{\Kbar})=\End(A_L)$.    Since $\G_{A_L,\ell}=\G_{A,\ell}$, we obtain an isomorphism between their commutants; $\End(A_L)\otimes_\ZZ \QQ_\ell = \End(A)\otimes_\ZZ \QQ_\ell$.    By comparing dimensions, we find that the injective map $\End(A)\otimes_\ZZ \QQ \to \End(A_L)\otimes_\ZZ \QQ=\End(A_{\Kbar})\otimes_\ZZ \QQ$ is an isomorphism.
\end{proof}

The following result of Bogomolov \cite{MR574307} says that the image of $\rho_{A,\ell}$ in $\G_{A,\ell}$ is large.  

\begin{prop} \label{P:Bogomolov}
The group $\rho_{A,\ell}(\Gal_K)$ is an open subgroup of $\G_{A,\ell}(\QQ_\ell)$ with respect to the $\ell$-adic topology.
\end{prop}

Using that $\rho_{A,\ell}(\Gal_K)$ is an open and compact subgroup of $\G_{A,\ell}(\QQ_\ell)$, we find that the algebraic group $\G_{A,\ell}$ describes the image of $\rho_{A,\ell}$ up to commensurability.  

\begin{prop} \label{P:Phi rank}
Assume that $K_A^\conn=K$.
\begin{romanenum}
\item The rank of the reductive group $\G_{A,\ell}$ does not depend on $\ell$.
\item Let $r$ be the common rank of the groups $\G_{A,\ell}$.  Then the set of $v\in\Sigma_K$ for which $\Phi_{A_v}$ is a free abelian of rank $r$ has density 1.
\end{romanenum}
\end{prop}
\begin{proof}
Fix a prime $\ell$.  For a place $v$ where $A$ has good reduction, let $\T_v$ be the Zariski closure in $\G_{A,\ell}$ of a fixed (semisimple) element in the conjugacy class $\rho_{A,\ell}(\Frob_v)$.    The set of places $v$ for which $\T_v$ is a maximal torus of $\G_{A,\ell}$ has density 1; this follows from \cite{MR1441234}*{Theorem~1.2}.   Observe that $\T_v$ is a maximal torus of $\G_{A,\ell}$ if and only if $\Phi_{A_v}$ (the multiplicative group generated by the eigenvalues of $\rho_{A,\ell}(\Frob_v)$) is a free abelian group whose rank equals the reductive rank of $\G_{A,\ell}$.      Parts (i) and (ii) follow since the rank of $\Phi_{A_v}$ does not depend on $\ell$.
\end{proof}

\subsection{The set $\calS_A$} \label{SS:calSA}

We define $\calS_A$ to be the set of places $v\in \Sigma_K$ that satisfy the following conditions:
\begin{itemize}
\item $A$ has good reduction at $v$,
\item $\FF_v$ has prime cardinality,
\item $\Phi_{A_v}$ is a free abelian group whose rank equals the common rank of the groups $\G_{A,\ell}$.
\end{itemize}
Using Proposition~\ref{P:Phi rank}(ii), one finds that $\calS_A$ has density 1 if $K_A^\conn=K$.  Since we are willing to exclude a set of places with density 0 from our main theorems, it will suffice to restrict our attention to the places $v\in\calS_A$.

\subsection{The Mumford-Tate group} \label{SS:MT group}
Fix a field embedding $K_A^\conn \subseteq \CC$.  The homology group $V=H_1(A(\CC),\QQ)$ is a vector space of dimension $2\dim A$ over $\QQ$.   It is naturally endowed with a $\QQ$-Hodge structure of type $\{(-1,0),(0,-1)\}$, and hence a decomposition 
\[
V\otimes_\QQ\CC = H_1(A(\CC),\CC)=V^{-1,0} \oplus V^{0,-1}
\]
such that $V^{0,-1}=\bbar{V^{-1,0}}$.  Let
\[
\mu\colon \GG_{m,\CC} \to \GL_{V\otimes_\QQ\CC}  
\]
be the cocharacter such that $\mu(z)$ is the automorphism of $V\otimes_\QQ\CC$ which is multiplication by $z$ on $V^{-1,0}$ and the identity on $V^{0,-1}$ for each $z\in\CC^\times=\GG_{m}(\CC)$.

\begin{definition}
The \defi{Mumford-Tate group} of $A$ is the smallest algebraic subgroup of $\GL_V$, defined over $\QQ$, which contains $\mu(\GG_{m,\CC})$.  We shall denote it by $\G_{A}$.
\end{definition}

The endomorphism ring $\End(A(\CC))$ acts on $V$; this action preserves the Hodge decomposition, and hence commutes with $\mu$ and thus also $\G_A$.  Moreover, the ring $\End(A(\CC))\otimes_\ZZ \QQ$ is natural isomorphic to the commutant of $\G_A$ in $\End_\QQ(V)$.   The group $\G_A/\QQ$ is reductive since the $\QQ$-Hodge structure for $V$ is pure and polarizable.  Using our fixed embedding $K_A^\conn\subseteq \CC$ and Proposition~\ref{P:Faltings}(ii), we have a natural isomorphism $\End(A(\CC))\otimes_\ZZ \QQ = \End(A_{K_A^\conn})\otimes_\ZZ \QQ$.

\subsection{The Mumford-Tate conjecture}  \label{SS:MT conj}

The \defi{comparison isomorphism} $V_\ell(A)\cong V \otimes_\QQ \QQ_\ell$ induces an isomorphism $\GL_{V_\ell(A)} \cong \GL_{V,\,\QQ_\ell}$.   The following conjecture says that $\G_{A,\ell}^\circ$ and $\G_{A,\QQ_\ell}$ are the same algebraic group when we use the comparison isomorphism as an identification, cf.~\cite{MR0476753}*{\S3}.

\begin{conj}[Mumford-Tate conjecture]   \label{C:MT}
For each prime $\ell$, we have $\G_{A,\ell}^\circ= \G_{A,\QQ_\ell}$.
\end{conj}
 
The Mumford-Tate conjecture is still open, however significant progress has been made in showing that several general classes of abelian varieties satisfy the conjecture; we simply refer to \cite{MR2400251}*{\S1.4} for a partial list of references.  The Mumford-Tate conjecture for $A$ holds if and only if the common rank of the groups $\G_{A,\ell}^\circ$ equals the rank of $\G_{A}$ \cite{MR1339927}*{Theorem~4.3}; in particular, the conjecture holds for one prime $\ell$ if and only if it holds for all $\ell$.  One inclusion of the Mumford-Tate conjecture is know to hold unconditionally, see Deligne's proof in \cite{MR654325}*{I, Prop.~6.2}.

\begin{prop}\label{P:MT inclusion}
For each prime $\ell$, we have $\G_{A,\ell}^\circ \subseteq  \G_{A,\QQ_\ell}$.
\end{prop}

Using this proposition, we obtain a well-defined Galois representation $\rho_{A,\ell} \colon \Gal_{K^\conn_A} \to  \G_{A}(\QQ_\ell)$ for each prime $\ell$.

\subsection{Frobenius conjugacy classes} \label{SS:Frobenius conjugacy classes}

Let $R$ be the affine coordinate ring of $\G_A$.   The group $\G_A$ acts on $R$ by composition with inner automorphisms.  We define $R^{\G_A}$ to be the $\QQ$-subalgebra of $R$ consisting of those elements fixed by this $\G_A$-action; it is the algebra of \defi{central functions} on $\G_A$.  Define $\Conj(\G_A):=\Spec(R^{\G_A})$; it is a variety over $\QQ$ which we call the \defi{variety of (semi-simple) conjugacy classes} of $\G_A$.  We define $\cl_{\G_A} \colon \G_A \to \Conj(\G_A)$ to be the morphism arising from the inclusion $R^{\G_A} \hookrightarrow R$ of $\QQ$-algebras.  

Let $L$ be an algebraically closed extension of $\QQ$.   Each $g\in \G_A(L)$ can be expressed uniquely in the form $g_s g_u$ where $g_s$ is semisimple, $g_u$ is unipotent, and $g_s$ and $g_u$ commute.    For $g,h\in \G_A(L)$, we have $g_s=h_s$ if and only if $\cl_{\G_A}(g)=\cl_{\G_A}(h)$.\\

Assume that $K_A^\conn=K$.  We can then view $\rho_{A,\ell}$ as having image in $\G_A(\QQ_\ell)$.     The following conjecture says that the conjugacy class of $\G_A$ containing $\rho_{A,\ell}(\Frob_v)$ does not depend on $\ell$; see \cite{MR1265537}*{C.3.3} for a more refined version.

\begin{conj}   \label{C:Frob conj}   Suppose that $K_A^\conn=K$.   Let $v$ be a finite place of $K$ for which $A$ has good reduction. Then there exists an $F_v \in \Conj(\G_A)(\QQ)$ such that $\cl_{\G_A}(\rho_{A,\ell}(\Frob_v))=F_v$ for all primes $\ell$ satisfying $v\nmid \ell$.
\end{conj}

\begin{remark}
The algebra of class functions of $\GL_V$ is $\QQ[a_1,\ldots, a_n]$ where the $a_i$ are the morphisms of $\GL_V$ that satisfy $\det(x I - g)=x^n+a_1(g)x^{n-1} + \ldots + a_{n-1}(g)x + a_n(g) 
$ for $g\in \GL_V(\QQ)$.   The inclusion $\G_A \subseteq \GL_V$ induces a morphism $f\colon \Conj(\G_A) \to \Conj(\GL_V):= \Spec \QQ[a_1,\ldots, a_n]\cong \AA_\QQ^n$.  Let $v$ be a finite place of $K$ for which $A$ has good reduction.  Conjecture~\ref{C:Frob conj} implies that for any prime $\ell$ satisfying $v\nmid\ell$, $f(\cl_{\G_A}(\rho_{A,\ell}(\Frob_v)))=f(F_v)$ belongs to $\Conj(\GL_V)(\QQ)$ and is independent of $\ell$; this consequence is true, and is just another way of saying that $\det(xI-\rho_{A,\ell}(\Frob_v))$ has coefficients in $\QQ$ and is independent of $\ell$.  
\end{remark}

In \S\ref{S:Noot}, we will state a theorem of Noot which gives a weakened version of Conjecture~\ref{C:Frob conj}.

\subsection{Image module $\ell$} \label{SS:image mod ell}
Let $\GL_{T_\ell(A)}$ be the group scheme over $\ZZ_\ell$ for which $\GL_{T_\ell(A)}(R)=\Aut_R(R \otimes_{\ZZ_\ell}T_\ell(A))$ for all (commutative) $\ZZ_\ell$-algebras $R$.   Note that the generic fiber of $\GL_{T_\ell(A)}$ is $\GL_{V_\ell(A)}$ and the image of $\rho_{A,\ell}$ lies in $\GL_{T_\ell(A)}(\ZZ_\ell)$.  Let $\calG_{A,\ell}$ be the Zariski closure of $\rho_{A,\ell}(\Gal_K)$ in $\GL_{T_\ell(A)}$; it is a group scheme over $\ZZ_\ell$ with generic fiber $\G_{A,\ell}$.   

Let $\bbar{\rho}_{A,\ell}\colon \Gal_K \to \Aut_{\ZZ/\ell\ZZ}(A[\ell])$ be the representation describing the Galois action on the $\ell$-torsion points of $A$.   Observe that $\bbar\rho_{A,\ell}(\Gal_K)$ is naturally a subgroup of $\calG_{A,\ell}(\FF_\ell)$; the following results show that these groups are almost equal.

\begin{prop} \label{P:image of Galois mod ell}
Suppose that $K_A^\conn=K$ and that $A$ is absolutely simple.
\begin{romanenum}
\item 
For $\ell$ sufficiently large, $\calG_{A,\ell}$ is a reductive group over $\ZZ_\ell$.
\item
There is a constant $C$ such that the inequality $[\calG_{A,\ell}(\FF_\ell): \bbar\rho_{A,\ell}(\Gal_K) ] \leq C$ holds for all primes $\ell$. 
\item 
For $\ell$ sufficiently large, the group $\bbar\rho_{A,\ell}(\Gal_K)$ contains the commutator subgroup of $\calG_{A,\ell}(\FF_\ell)$.
\end{romanenum}
\end{prop}
\begin{proof}
In Serre's 1985-1986 course at the Coll\`ege de France \cite{MR1730973}*{136}, he showed that the groups $\bbar\rho_{A,\ell}(\Gal_K)$ are essentially the $\FF_\ell$-points of certain reductive groups.   For each prime $\ell$, he constructs a certain connected algebraic subgroup $H_\ell$ of $\GL_{T_\ell(A),\FF_\ell}\cong \GL_{2\dim A,\FF_\ell}$.   There exists a finite extension $L/K$ for which the following properties hold for all sufficiently large primes $\ell$:
\begin{itemize}
\item  $H_\ell$ is reductive,
\item $\bbar\rho_{A,\ell}(\Gal_L)$ is a subgroup of $H_\ell(\FF_\ell)$ and the index $[H_\ell(\FF_\ell):\bbar\rho_{A,\ell}(\Gal_L)]$ can be bounded independent of $\ell$.
\item $\bbar\rho_{A,\ell}(\Gal_L)$ contains the commutator subgroup of $H_\ell(\FF_\ell)$.
\end{itemize}
Detailed sketches of Serre's results were supplied in letters that have since been published in his collected papers; see the beginning of \cite{MR1730973}, in particular the letter to M.-F.~Vign\'eras \cite{MR1730973}*{137}.  The paper \cite{MR1944805} also contains everything we need.

In \cite{MR1944805}*{\S3.4}, it is shown that Serre's group $H_\ell$ equals the special fiber of $\calG_{A,\ell}$ for all sufficiently large $\ell$.  Parts (ii) and (iii) then follow from the properties of $H_\ell$.  For part (i), see \cite{MR1944805}*{\S2.1} and \cite{MR1339927}.
\end{proof}

\subsection{Independence}
Combining all our $\ell$-adic representations together, we obtain a single Galois representation
\[
\rho_A \colon \Gal_K \to \prod_\ell \Aut_{\QQ_\ell}(V_\ell(A))
\]
which describes the Galois action on all the torsion points of $A$.  The following theorem shows that, after possibly replacing $K$ by a finite extension, the Galois representations $\rho_{A,\ell}$ will be \emph{independent}.   
\begin{prop}[Serre \cite{MR1730973}*{138}] \label{P:independence}
There is a finite Galois extension $K'$ of $K$ in $\Kbar$ such that $\rho_A(\Gal_{K'})$ equals $\prod_\ell \rho_{A,\ell}(\Gal_{K'})$.
\end{prop}

We will need the following straightforward consequence:

\begin{prop}  \label{P:independence density}
Fix an extension $K'/K$ as in Proposition~\ref{P:independence}.   Let $\Lambda$ be a finite set of rational primes.   For each prime $\ell \in \Lambda$, fix a subset $U_\ell$ of $\bbar\rho_{A,\ell}(\Gal_K)$ that is stable under conjugation.   Let $\calS$ be the set of $v\in \Sigma_K$ such that $\bbar\rho_{A,\ell}(\Frob_v)\subseteq U_\ell$ for all $\ell\in \Lambda$.   Then $\calS$ has density
\[
\sum_{C} \frac{|C|}{|\Gal(K'/K)|}\cdot \prod_{\ell\in \Lambda} \frac{|\bbar\rho_{A,\ell}(\Gamma_C) \cap U_\ell |}{|\bbar\rho_{A,\ell}(\Gamma_C)|}
\]
where $C$ varies over the conjugacy classes of $\Gal(K'/K)$ and $\Gamma_C$ is the set of $\sigma\in \Gal_K$ for which $\sigma|_{K'}\in C$.
\end{prop}
\begin{proof}
Set $m:=\prod_{\ell\in \Lambda} \ell$, and define $U_m:=\prod_{\ell|m} U_\ell$ which we view as a subset of $\Aut_{\ZZ/m\ZZ}(A[m])$.   Let $\bbar\rho_{A,m}\colon \Gal_K\to \Aut_{\ZZ/m\ZZ}(A[m])$ be the homomorphism describing the Galois action on $A[m]$.  Let $\mu$ and $\mu'$ be the Haar measures on $\Gal_K$ normalized so that $\mu(\Gal_K)=1$ and $\mu'(\Gal_{K'})=1$.    

The Chebotarev density theorem says that the density $\delta$ of $\calS$ is defined and equals $\mu(\{\sigma \in \Gal_K: \bbar\rho_{A,m}(\sigma)\in U_m \})$.   Let $\{\sigma_i\}_{i\in I}$ be a subset of $\Gal_K$ consisting of representatives of the cosets of $\Gal_{K'}$ in $\Gal_K$.  We then have
\begin{align*}
[K':K]\delta=\sum_{i\in I} \mu'(\{\sigma \in \sigma_i\Gal_{K'}: \bbar\rho_{A,m}(\sigma)\in U_m\}) = \sum_{i\in I} \frac{|\bbar\rho_{A,m}(\sigma_i\Gal_{K'})\cap U_m|}{|\bbar\rho_{A,m}(\sigma_i\Gal_{K'})|}.
\end{align*}
We have $\bbar\rho_{A,m}(\Gal_{K'}) = \prod_{\ell|m} \bbar\rho_{A,\ell}(\Gal_{K'})$ by our choice of $K'$ and hence $\bbar\rho_{A,m}(\sigma_i\Gal_{K'})$ equals $\prod_{\ell|m} \bbar\rho_{A,\ell}(\sigma_i\Gal_{K'})$ for all $i\in I$.  Therefore,
\begin{align*}
[K':K]\delta= \sum_{i\in I} \prod_{\ell|m} \frac{|\bbar\rho_{A,\ell}(\sigma_i\Gal_{K'})\cap U_\ell|}{|\bbar\rho_{A,\ell}(\sigma_i\Gal_{K'})|}
\end{align*}
Since $U_\ell$ is stable under conjugation, we find that $|\bbar\rho_{A,\ell}(\sigma_i\Gal_{K'})\cap U_\ell|/|\bbar\rho_{A,\ell}(\sigma_i\Gal_{K'})|$ depends only on the conjugacy class $C$ of $\Gal(K'/K)$ containing $\sigma_i|_{K'}$ and equals ${|\bbar\rho_{A,\ell}(\Gamma_C) \cap U_\ell |}/{|\bbar\rho_{A,\ell}(\Gamma_C)|}$.   Using this and grouping the $\sigma_i$ by their conjugacy class when restricted to $K'$, we deduce that  
\[
[K':K]\delta= \sum_{C} |C| \prod_{\ell | m} \frac{|\bbar\rho_{A,\ell}(\Gamma_C) \cap U_\ell |}{|\bbar\rho_{A,\ell}(\Gamma_C)|}
\]
where $C$ varies over the conjugacy classes of $\Gal(K'/K)$.
\end{proof}

\section{Reductive groups: background}  \label{S:reductive}
Fix a perfect field $k$ and an algebraic closure $\kbar$.  
\subsection{Tori} 
An \defi{(algebraic) torus} over $k$ is an algebraic group $\T$ defined over $k$ for which $\T_{\kbar}$ is isomorphic to $\GG_{m,\kbar}^r$ for some integer $r$.   Fix a torus $\T$ over $k$.    Let $X(\T)$ be the group of characters $\T_{\kbar} \to \GG_{m,\kbar}$; it is a free abelian group whose rank equals the dimension of $\T$.   Let $\Aut(\T_{\kbar})$ be the group of automorphisms of the algebraic group $\T_{\kbar}$.      For each $f\in \Aut(\T_{\kbar})$, we have an isomorphism $f_*\colon X(\T)\to X(\T), \,\alpha\mapsto \alpha\circ f^{-1}$; this gives a group isomorphism 
\[
\Aut(\T_{\kbar}) \xrightarrow{\sim} \Aut(X(\T)),\,\, f\mapsto f_*
\]
that we will often use as an identification.

There is a natural action of the absolute Galois group $\Gal_k$ on $X(\T)$; it satisfies $\sigma(\alpha(t))=\sigma(\alpha)(\sigma(t))$ for all $\sigma\in \Gal_k$, $\alpha\in X(\T)$ and $t\in \T(\kbar)$.  Let $\varphi_\T\colon \Gal_k \to \Aut(X(\T))$ be the homomorphism describing this action, that is, $\varphi_\T(\sigma)\alpha = \sigma(\alpha)$ for $\sigma\in \Gal_k$ and $\alpha\in X(\T)$.   We say that the torus $\T$ is \defi{split} if it is isomorphic to $\GG_{m,k}^r$; equivalently, if $\varphi_\T(\Gal_k)=1$.\\

Let $\G$ be a reductive group over $k$.   A \defi{maximal torus} of $\G$ is a closed algebraic subgroup that is a torus (also defined over $k$) and is not contained in any larger such subgroup.   If $\T$ and $\T'$ are maximal tori of $\G$, then $\T_{\kbar}$ and $\T'_{\kbar}$ are maximal tori of $\G_{\kbar}$ and are conjugate by some element of $\G(\kbar)$.  The group $\G$ has a maximal torus whose dimension is called the \defi{rank} of $\G$.  We say that $\G$ is \defi{split} if it has a maximal torus that is split.

\subsection{Weyl group} \label{SS:Weyl group}

Let $\G$ be a connected reductive group over $k$.   Fix a maximal torus $\T$ of $\G$.  The \defi{Weyl group} of $\G$ with respect to $\T$ is the (finite) group
 \[
 W(\G,\T) := N_{\G}(\T)(\kbar)/\T(\kbar)
 \]
 where $N_{\G}(\T)$ is the normalizer of $\T$ in $\G$.      For an element $g\in N_{\G}(\T)(\kbar)$, the morphism $\T_{\kbar}\to \T_{\kbar},\, t\mapsto gtg^{-1}$ is an isomorphism that depends only on the image of $g$ in $W(\G,\T)$; this induces a faithful action of $W(\G,\T)$ on $\T_{\kbar}$.  So we can identify $W(\G,\T)$ with a subgroup of $\Aut(\T_{\kbar})$ and hence also of $\Aut(X(\T))$.  
 
  There is a natural action of $\Gal_k$ on $W(\G,\T)$.  For $\sigma\in \Gal_k$ and $w\in W(\G,\T)$, we have  $\varphi_\T(\sigma) \circ w \circ \varphi_\T(\sigma)^{-1}=\sigma(w)$.  In particular, note that the action of $\Gal_k$  of $W(\G,\T)$ is trivial if $\T$ is split.
  
We define $\Pi(\G,\T)$ to be subgroup of $\Aut(X(\T))$ generated by $W(\G,\T)$ and $\varphi_\T(\Gal_k)$.  The Weyl group $W(\G,\T)$ is a normal subgroup of $\Pi(\G,\T)$.   Up to isomorphism, the groups $W(\G,\T)$ and $\Pi(\G,\T)$ are independent of $\T$; we shall denote the abstract groups by $W(\G)$ and $\Pi(\G)$, respectively.  
 
\subsection{Maximal tori over finite fields} \label{SS:tori over finite fields}
We now assume that $k$ is a finite field $\FF_q$ with $q$ elements.   Let $\G$ be a connected reductive group defined over $\FF_q$.   Assume further that $\G$ is split and fix a split maximal torus $\T$.   Let $\T'$ be any maximal torus of $\G$.    There is an element $g\in \G(\FFbar_q)$ such that $g\T'_{\FFbar_q} g^{-1}=\T_{\FFbar_q}$.   Since $\T$ and $\T'$ are defined over $\FF_q$, we find that $\Frob_q(g)\T'_{\FFbar_q} \Frob_q(g)^{-1}=\T_{\FFbar_q}$ and hence $g \Frob_q(g)^{-1}$ belongs to $N_{\G}(\T)(\FFbar_q)$.    Let $\theta_{\G_A}(\T')$ be the conjugacy class of $W(\G,\T)$ containing the coset represented by $g \Frob_q(g)^{-1}$.   These conjugacy classes have the following interpretation:

\begin{prop}
The map $\T'\mapsto \theta_{\G}(\T')$ defines a bijection between the maximal tori of $\G$ up to conjugation in $\G(\FF_q)$ and the conjugacy classes of $W(\G,\T)$.
\end{prop}
\begin{proof}
This is \cite{MR794307}*{Prop.~3.3.3}.  Note that the action of $\Frob_q$ on $W(\G,\T)$ is trivial since $\T$ is split, so the $\Frob_q$-conjugacy classes of $W(\G,\T)$ in \cite{MR794307} are just the usual conjugacy classes of $W(\G,\T)$.
\end{proof}

Let $\G(\FF_q)_{sr}$ be the set of $g\in \G(\FF_q)$ that are semisimple and regular in $\G$.   Each $g\in \G(\FF_q)_{sr}$ is contained in a unique maximal torus $\T_g$ of $\G$.   We define the map 
\[
\theta_{\G} \colon \G(\FF_q)_{sr} \to W(\G,\T)^\sharp,\quad g\mapsto \theta_{\G}(\T_g)
\]
where $W(\G,\T)^\sharp$ is the set of conjugacy classes of $W(\G,\T)$.  We will need the following equidistribution result later.

\begin{lemma} \label{L:JKZ estimate}
Let $\G$ be a connected and split reductive group over a finite field $\FF_q$, and fix a split maximal torus $\T$.   Let $C$ be a subset of $W(\G,\T)$ that is stable under conjugation and let $\kappa$ be a subset of $\G(\FF_q)$ that is a union of cosets of the commutator subgroup of $\G(\FF_q)$.  Then
\[
\frac{|\{ g \in \kappa \cap \G(\FF_q)_{sr} : \theta_{\G}(g)\subseteq C\}|}{|\kappa|} = \frac{|C|}{|W(\G,\T)|} + O(1/q)
\]
where the implicit constant depends only on the type of $\G$.
\end{lemma}
\begin{proof}
We shall reduce to a special case treated in \cite{JKV-splitting} which deals with semisimple groups.  Let $\G^{\ad}$ be the quotient of $\G$ by its center and let $\varphi\colon \G\to\G^{\ad}$ be the quotient homomorphism.   Let $\T^{\ad}$ be the image of $\T$ under $\varphi$; it is a split maximal torus of $\G^{\ad}$.   The homomorphism $\varphi$ induces a group isomorphism $\varphi_*\colon W(\G,\T)\xrightarrow{\sim} W(\G^{\ad},\T^{\ad})$.  An element $g\in \G(\FFbar_q)$ is regular and semisimple in $\G$ if and only if $\varphi(g)$ is regular and semisimple in $\G^{\ad}$.   For $g\in \G(\FF_q)_{sr}$, one can check that $\theta_{\G}(g)\subseteq C$ if and only if $\theta_{\G^{\ad}}(\varphi(g))\subseteq \varphi_*(C)$.   Therefore,
\begin{equation} \label{E:reduce to adjoint case}
\frac{|\{ g \in \kappa \cap \G(\FF_q)_{sr} : \theta_{\G}(g)\subseteq C\}|}{|\kappa|} = \frac{|\{ g \in \varphi(\kappa) \cap \G^{\ad}(\FF_q)_{sr} : \theta_{\G^{\ad}}(g)\subseteq \varphi_*(C)\}|}{|\varphi(\kappa)|}.
\end{equation}

Let $\G^{\der}$ be the derived subgroup of $\G$ and let $\pi\colon \G^{\sc}\to \G^{\der}$ be the simply connected cover of $\G^\der$.   The homomorphism $\pi':=\varphi\circ \pi\colon \G^{\sc} \to \G^{\ad}$ is a simply connected cover of $\G^{\ad}$.  Since $\G^{\ad}$ is adjoint, it is the product of simple adjoint groups defined over $\FF_q$.   Assuming that $q$ is sufficiently large, the group $\pi'(\G^{\sc}(\FF_q))$ agrees with the commutator subgroup of $\G^\ad(\FF_q)$ and is a product of simple groups of Lie type, see \cite{MR1370110}*{\S2.1} for background.  (We can later choose the implicit constant in the lemma to deal with the finitely many excluded $q$.)   Since $\pi'(\G^{\sc}(\FF_q))$ is perfect, we find that the image of the commutator subgroup of $\G(\FF_q)$ under $\varphi$ is $\pi'(\G^{\sc}(\FF_q))$.  In particular, $\varphi(\kappa) \subseteq \G^{\ad}(\FF_q)$ consists of cosets of $\pi'(\G^{\sc}(\FF_q))$.  Proposition 4.6 of \cite{JKV-splitting} now applies, and shows that that the right hand side of (\ref{E:reduce to adjoint case}) equals $|\varphi_*(C)|/|W(\G^{\ad},\T^{\ad})| +O(1/q) = |C|/|W(\G,\T)| + O(1/q)$ where the implicit constants depend only on the type of $\G^{\ad}$; the proposition is only stated for a single $\pi'(\G^{\sc}(\FF_q))$ coset, but one observes that the index $[\G^{\ad}(\FF_q):\pi'(\G^{\sc}(\FF_q))]$ can be bounded in terms of the type of $\G^{\ad}$.
\end{proof}

\section{Frobenius conjugacy classes} \label{S:Noot}
Let $A$ be a non-zero abelian variety over a number field $K$.   Assume that $K_A^\conn=K$ and fix an embedding $K\subseteq \CC$.   In this section, we state a theorem of R.~Noot which gives a weakened version of Conjecture~\ref{C:Frob conj}.

\subsection{The variety $\textrm{Conj}'(\G_A)$}
We first need to define a variant of the variety $\Conj(\G_A)$ from \S\ref{SS:Frobenius conjugacy classes}.  Let $\G_A^\der$ be the derived subgroup of $\G_A$.   Let $\{\mathbf{H}_i\}_{i\in I}$ be the minimal non-trivial normal connected closed subgroups of $(\G_A^\der)_{\Qbar}$.   The groups $\mathbf{H}_i$ are semisimple.  The morphism $\prod_{i\in I} \mathbf{H}_i \to (\G_A^\der)_{\Qbar},\, (g_i)_{i\in I}\mapsto \prod_{i\in I} g_i$  is a homomorphism of algebraic groups and has finite kernel.  

 Let $J$ be the set of $i\in I$ for which $\mathbf{H}_i$ is isomorphic to $\SO(2k_i)_{\Qbar}$ for some integer $k_i\geq 4$.  For each $i\in J$, we identify $\mathbf{H}_i$ with $\SO(2k_i)_{\Qbar}$ and we set $\mathbf{H}_i':=\mathbf{O}(2k_i)_{\Qbar}$.   For $i\in I-J$, we set $\mathbf{H}_i':=\mathbf{H}_i$.

Let $\mathbf{C}$ be the center of $\G_A$.  We define $\calA$ to be the group of automorphisms $f$ of the algebraic group $\G_{A,\Qbar}$ which satisfy the following properties:
\begin{itemize} 
\item $f(\mathbf{H}_i)=\mathbf{H}_i$ for all $i\in I$,
\item the morphism $f|_{\mathbf{H}_i}\colon \mathbf{H}_i \to \mathbf{H}_i$ agrees with conjugation by some element in $\mathbf{H}_i'$,
\item the morphism $f|_{\mathbf{C}_\Qbar}\colon \mathbf{C}_\Qbar\to \mathbf{C}_\Qbar$ is the identity map.
\end{itemize}

Let $R$ be the affine coordinate ring of $\G_A$.   The group $\calA$ acts on $R$ by composition, and we define $R^{\calA}$ to be the $\QQ$-subalgebra of $R$ consisting of those elements fixed by the $\calA$-action.   Define the $\QQ$-variety $\Conj'(\G_A):=\Spec(R^{\calA})$ and let $\cl_{\G_A}' \colon \G_A \to \Conj'(\G_A)$ be the morphism arising from the inclusion $R^{\calA} \hookrightarrow R$ of $\QQ$-algebras.  

\subsection{A theorem of Noot}
By Proposition~\ref{P:MT inclusion} and our ongoing assumption $K_A^\conn=K$, the representation $\rho_{A,\ell}$ has image in $\G_A(\QQ_\ell)$. 
The following is a consequence of \cite{MR2472133}*{Th\'eor\`eme~1.8}.

\begin{thm}[Noot] \label{T:Noot}
Let $v$ be a finite place of $K$ for which $A$ has good reduction.  Suppose that $\pi_1\pi_2^{-1}$ is not a root of unity for all distinct roots $\pi_1,\pi_2\in \Qbar$ of $P_{A_v}(x)$.  Then there exists an $F_v' \in \Conj'(\G_A)(\QQ)$ such that $F_v'=\cl_{\G_A}'(\rho_{A,\ell}(\Frob_v))$ for all primes $\ell$ satisfying $v\nmid \ell$.
\end{thm}

The group of inner automorphisms of $\G_{A,\Qbar}$ is a normal subgroup of finite index in $\calA$.   So each element of $R^{\calA}$ is a central function of $\G_{A}$, and we have a natural morphism $\varphi\colon \Conj(\G_A)\to \Conj'(\G_A)$ that satisfies $\cl_{\G_A}'=\varphi\circ \cl_{\G_A}$.   Observe that if Conjecture~\ref{C:Frob conj} holds, then the $F_v'$ in Noot's theorem equals $\varphi(F_v)$.

\subsection{The group $\Gamma$}
Fix a maximal torus $\T$ of $\G_A$.    Let $\calA(\T)$ be the subgroup of $f\in \calA$ that satisfy $f(\T_{\Qbar})=\T_{\Qbar}$.  Every element of $\calA$ is conjugate to an element of $\calA(\T)$ by an inner automorphism of $\G_{A,\Qbar}$.   Define  
\[
\Gamma := \{ f|_{\T_{\Qbar}} : f \in \calA(\T) \};
\]
it is a (finite) subgroup of $\Aut(\T_{\Qbar})$ which is stable under the action of $\Gal_{\QQ}$.   For $t_1,t_2 \in \T(\Qbar)$, we have $\cl_{\G_A}'(t_1)=\cl_{\G_A}'(t_2)$ if and only if $t_2=\beta(t_1)$ for some $\beta\in \Gamma$.    So using $\cl_{\G_A}'$, we find that the variety $\Conj'(\G_A)_\Qbar$ is the quotient of the torus $\T_\Qbar$ by $\Gamma$.

Observe that $W(\G_A,\T)$ is a normal subgroup of $\Gamma$.    The following technical lemma will be important later.

\begin{lemma}  \label{L:new Jordan}
Suppose $H$ is a subgroup of $\Gamma$ such that $H\cap C \neq \emptyset$ for each conjugacy class $C$ of $\Gamma$ contained in $W(\G_A,\T)$.   Then $H\supseteq W(\G_A,\T)$.
\end{lemma}
\begin{proof}
Since $W(\G_A,\T)$ is a normal subgroup of $\Gamma$, there no harm in replacing $H$ by $H\cap W(\G_A,\T)$; thus without loss of generality, we may assume that $H$ is a subgroup of $W(\G_A,\T)$.

Let $\Phi:=\Phi(\G_A,\T)\subseteq X(\T)$ be the set of roots of $\G_A$ with respect to $\T$, cf.~\cite{MR1102012}*{\S8.17}.  The set of roots with the embedding $\Phi \hookrightarrow X(\T/\mathbf{C})\otimes_\ZZ \RR$ form an abstract root system.  The root system $\Phi$ is the disjoint union of its irreducible components $\{\Phi_i\}_{i\in I}$, where the root systems $\Phi_i$ correspond with our subgroups $\mathbf{H}_i$.    

We can identify $\Gamma$ with a subgroup of $\Aut(X(\T))$.  For $f\in \Gamma$, we have $f(\Phi)=\Phi$; moreover, $f(\Phi_i)=\Phi_i$ for $i\in I$.   For each $i\in I$, we have a homomorphism $\Gamma\to \Aut(\Phi_i), f\mapsto f|_{\Phi_i}$ whose image we denote by $\Gamma_i$.   Let $W(\Phi_i)$ be the Weyl group of $\Phi_i$; it is a subgroup of index at most 2 in $\Gamma_i$.  If $i\notin J$, then we have $\Gamma_i=W(\Phi_i)$.  The group $W(\G_A,\T)$ acts faithfully on $\Phi$, and one can then check that $\Gamma$ also acts faithfully on $\Phi$.   Therefore, the natural map $\Gamma \to \prod_{i\in I} \Gamma_i \subseteq \Aut(\Phi)$ is injective and $W(\G_A,\T)$ is mapped to the Weyl group $\prod_{i\in I}W(\Phi_i) =W(\Phi)$.  We may thus identify $H$ with a subgroup of $\prod_{i\in I} \Gamma_i$.  Let $H_i$ be the group of $(f_j)_{j\in I} \in \prod_{j\in I} \Gamma_j$ which belong to $H$ and satisfy $f_j=1$ for $j\neq i$.    We have $\prod_{i\in I}H_i \subseteq H\subseteq W(\Phi)$, so it suffices to show that $W(\Phi_i)\subseteq H_i$ for every $i\in I$.   

Fix any $i\in I$.  From our assumptions on $H$, we find that $H_i$ is a subgroup of $W(\Phi_i)$ such that $H_i \cap C\neq \emptyset$ for every conjugacy classes $C$ of $\Gamma_i$ that is contained in $W(\Phi_i)$.     If $\Gamma_i=W(\Phi_i)$, then we have $H_i=W(\Phi_i)$ by Jordan's lemma.       It remains to consider the case where $\Gamma_i\neq W(\Phi)$.\\ 

We have reduced the lemma to the following situation:  Let $\Phi$ be an irreducible root system of type $D_n$ with $n\geq 4$.  Let $\Gamma$ be a subgroup of $\Aut(\Phi)$ that contains $W(\Phi)$ and satisfies $[\Gamma:W(\Phi)]=2$.   Let $H$ be a subgroup of $W(\Phi)$ that satisfies $H\cap C\neq \emptyset$ for all conjugacy classes $C$ of $\Gamma$ contained in $W(\Phi)$.   We need to show that $H=W(\Phi)$.

We can identify the root system $\Phi$ with the set of vectors $\pm e_i  \pm e_j$ with $1\leq i< j\leq n$ in $\RR^n$, where $e_1,\ldots, e_n$ is the standard basis of $\RR^n$.   Let $\Gamma'$ be the group of automorphisms $f$ of the vector space $\RR^n$ such that for each $1\leq i \leq n$, we have $f(e_i)=\varepsilon_i e_{j}$ for some $j\in\{1,\ldots, n\}$ and $\varepsilon_i\in\{\pm 1\}$.  Ignoring the signs, each $f\in \Gamma'$ gives a permutation of $\{1,\ldots, n\}$; this defines a short exact sequence
\[
1\to N \to \Gamma' \xrightarrow{\varphi} S_n \to 1
\]
where the group $N$ consists of those $f\in \Gamma'$ that satisfy $f(e_i)=\pm e_i$ for all $1\leq i \leq n$.   The Weyl group $W(\Phi)$ is the subgroup of index 2 in $\Gamma'$ consisting of those $f$ for which $\prod_{i=1}^n \varepsilon_i=1$.     We may assume that $\Gamma=\Gamma'$;  for $n>5$, this is because $\Gamma'=\Aut(\Phi)$ (for $n=4$, the subgroups of $\Aut(\Phi)$ that contain $W(\Phi)$ as an index subgroup of order 2 are all conjugate to $\Gamma'$).   Restricting $\varphi$ to $W(\Phi)$, we have a short exact sequence
\[
1\to N' \to W(\Phi) \xrightarrow{\varphi|_{W(\Phi)}} S_n \to 1
\]
where $N'$ is the group of $f\in N$ for which $f(e_i)=\varepsilon_i e_i$ and $\prod_i \varepsilon_i = 1$.  Since $\varphi(\Gamma)=\varphi(W(\Phi))=S_n$, our assumption on $H$ implies that $\varphi(H)\cap C \neq \emptyset $ for each conjugacy class $C$ of $S_n$.   We thus have $\varphi(H)=S_n$ by Jordan's lemma.  It thus suffices to prove that $H\supseteq N'$.   

For a subset $B\subseteq \{1,\ldots, n\}$ with cardinality 2, we let $f_B$ be the element of $N'$ for which $f_B(e_i)=-e_i$ if $i\in B$ and $f_B(e_i)=e_i$ otherwise.     For $g\in \Gamma$, we have $g f_B g^{-1} = f_{\sigma(B)}$ where $\sigma:=\varphi(g)$.   Therefore, $H$ contains an element of the form $f_B$ for some set $B\subseteq \{1,\ldots, n\}$ with cardinality 2 (such functions form a conjugacy class of $\Gamma$ in $W(\Phi)$).   Since $\varphi(H)=S_n$, we deduce that $H$ contains all the $f_B$ with $|B|=2$, and hence $H\supseteq N'$ (since $N'$ is generated by such $f_B$).
\end{proof} 

\section{Local representations}
\label{S:local Galois}

Fix a non-zero abelian variety $A$ defined over a number field $K$ such that $K^\conn_A=K$.        Fix a prime $\ell$ and suppose that $\calG_{A,\ell}$ is a reductive group scheme over $\ZZ_\ell$ which has a split maximal torus $\calT$.   Denote the generic fiber of $\calT$ by $\T$; it is a maximal torus of $\G_{A,\ell}$.\\

Take any place $v\in \calS_A$ that satisfies $v\nmid \ell$.  Define the set
\[
\mathcal{I}_{v,\ell}:= \big\{ t \in \T(\Qbar_\ell) : t \text{ and } \rho_{A,\ell}(\Frob_v) \text{ are conjugate in } \G_{A,\ell}(\Qbar_\ell)\big\}
\]
and fix an element $t_{v,\ell} \in \mathcal{I}_{v,\ell}$.  Conjugation induces an action of the Weyl group $W(\G_{A,\ell},\T)$ on $\mathcal{I}_{v,\ell}$.   Since $v$ belongs to $\calS_A$, we find that the group generated by $t_{v,\ell}$ is Zariski dense in $\T_{\Qbar_\ell}$ and hence the action of $W(\G_{A,\ell},\T)$ on $\mathcal{I}_{v,\ell}$ is simply transitive.

Since $\G_{A,\ell}$ and $\rho_{A,\ell}(\Frob_v)$ are defined over $\QQ_\ell$, we also have a natural action of $\Gal_{\QQ_\ell}$ on $\mathcal{I}_{v,\ell}$.    So for each $\sigma\in \Gal_{\QQ_\ell}$, there is a unique $\psi_{v,\ell}(\sigma)\in W(\G_{A,\ell},\T)$ that satisfies $\sigma(t_{v,\ell})=\psi_{v,\ell}(\sigma)^{-1}(t_{v,\ell})$.    Using that $\T$ is split, one can show that that map
\[
\psi_{v,\ell}\colon \Gal_{\QQ_\ell}\to W(\G_{A,\ell},\T),\quad \sigma\mapsto \psi_{v,\ell}(\sigma)
\]
is a group homomorphism.  Note that a different choice of $t_{v,\ell}$ would alter $\psi_{v,\ell}$ by an inner automorphism of $W(\G_{A,\ell},\T)$.   

Choose an embedding $\Qbar\subseteq \Qbar_\ell$.   The homomorphism $\psi_{v,\ell}$ then factors through an injective group homomorphism $\Gal(\QQ_\ell(\calW_{A_v})/\QQ_\ell) \hookrightarrow W(\G_{A,\ell},\T)$.

\begin{lemma}  \label{L:sieving input}
Fix a subset $C$ of $W(\G_{A,\ell},\T)$ that is stable under conjugation.    There is a subset $U_\ell$ of $\bbar\rho_{A,\ell}(\Gal_K)$ which is stable under conjugation and satisfies the following properties:
\begin{itemize}
\item If $v\in \calS_A$ satisfies $v\nmid \ell$ and $\bbar\rho_{A,\ell}(\Frob_v) \subseteq U_\ell$, then $\psi_{v,\ell}$ is unramified and $\psi_{v,\ell}(\Frob_\ell)\subseteq C$.
\item Let $K'$ be a finite extension of $K$ and let $\kappa$ be a coset of $\Gal_{K'}$ in $\Gal_K$.  Then we have
\[
\frac{|\bbar{\rho}_{A,\ell}(\kappa) \cap U_\ell|}{|\bbar{\rho}_{A,\ell}(\kappa)|} = \frac{|C|}{|W(\G_{A,\ell},\T)|} + O(1/\ell)
\]
where the implicit constant depends only on $A$ and $K'$.
\end{itemize}
\end{lemma}

\begin{proof}
Set $\calG:=\calG_{A,\ell}$; it is a reductive group scheme over $\ZZ_\ell$ by assumption.    The special fiber $\calG_{\FF_\ell}$ is a reductive group with split maximal torus $\calT_{\FF_\ell}$.   Assuming $\ell$ is sufficiently large, the derived subgroups of $\calG_{\FF_\ell}$ and $\calG_{\QQ_\ell}=\G_{A,\ell}$ are of the same Lie type; this follows from \cite{MR1944805}*{Th\'eor\`eme~2}.    Note that we can set $U_\ell=\emptyset$ for the finitely many excluded primes.   Therefore, the Weyl groups $W(\G_{A,\ell},\T)$ and $W(\calG_{\FF_\ell},\calT_{\FF_\ell})$ are abstractly isomorphic; we now describe an explicit isomorphism.  The homomorphism
\begin{equation} \label{E:Weyl isom 1}
    N_{\calG}(\calT )(\ZZ_\ell)/\calT (\ZZ_\ell)  \hookrightarrow  N_{\calG}(\calT )(\QQ_\ell)/\calT (\QQ_\ell) =
    W(\calG_{\QQ_\ell},\calT _{\QQ_\ell}) = W(\G_{A,\ell},\T)
\end{equation}
is injective; the identification with the Weyl group use that $\calT _{\QQ_\ell}=\T$ is split.  The normalizer $N_{\calG}(\calT )$ is a closed and smooth subscheme of $\calG$; for smoothness, see~\cite[XXII~Corollaire~5.3.10]{SGA3}.  The homomorphisms $N_{\calG}(\calT )(\ZZ_\ell)\to N_{\calG}(\calT )(\FF_{\ell})$ and $\calT (\ZZ_\ell)\to \calT (\FF_{\ell})$ are thus surjective by Hensel's lemma, and we obtain a  
 a surjective homomorphism
\begin{equation} \label{E:Weyl isom 2}
  N_{\calG}(\calT  )(\ZZ_{\ell})/\calT  (\ZZ_{\ell}) \twoheadrightarrow N_{\calG_{\FF_{\ell}}}(\calT_{\FF_{\ell}})(\FF_{\ell})/\calT_{\FF_{\ell}}(\FF_{\ell}) = W(\calG_{\FF_{\ell}},\calT_{\FF_{\ell}}).
\end{equation}
Since (\ref{E:Weyl isom 1}) and (\ref{E:Weyl isom 2}) are injective and surjective homomorphisms, respectively, into isomorphic groups, we deduce that they are both isomorphisms.  Combining the isomorphisms (\ref{E:Weyl isom 1}) and (\ref{E:Weyl isom 2}), we obtain the desired isomorphism
$W(\G_{A,\ell},\T) \xrightarrow{\sim} W(\calG_{\FF_{\ell}},\calT_{\FF_{\ell}})$.  

Now fix a place $v\in \calS_A$, and let $h\in \calG(\ZZ_\ell)$ be a representative of the conjugacy class $\rho_{A,\ell}(\Frob_v)$.   We know that $h$ is semisimple and regular in $\calG_{\QQ_\ell}=\G_{A,\ell}$.  Assume further that the image $\bbar{h}$ in $\calG(\FF_\ell)$ is semsimple and regular.   The centralizer $\calT_h$ of $h$ in $\calG$ is then a smooth and closed subscheme whose generic and special fibers are both maximal tori, i.e., $\calT_h$ is a maximal torus of $\calG$.   The transporter $\textrm{Transp}_{\calG}(\calT_h ,
\calT)$ is a closed and smooth group scheme in $\calG$; again for smoothness, see~\cite[XXII~Corollaire 5.3.10]{SGA3}.  Recall
that for any $\ZZ_\ell$-algebra $R$, we have
\[
\textrm{Transp}_{\calG}(\calT_h  , \calT )(R) = \{ g \in \calG(R) : g\,\calT_{h,R}\, g^{-1} = \calT_{R} \}.
\]
Choose any point $\overline{g} \in \textrm{Transp}_{\calG}(\calT_h,\calT )(\FFbar_{\ell})$.  Let $\ZZ_\ell^{un}$ be the ring of integers in the maximal unramified extension of $\QQ_\ell$ in $\Qbar_\ell$.  Since $\textrm{Transp}_{\calG}(\calT_h,\calT)$ is smooth and $\ZZ_{\ell}^{un}$ is Henselian, there is an $g \in \textrm{Transp}_{\calG}(\calT_h,\calT)(\ZZ_{\ell}^{un})$ which lifts $\overline{g}$.    The element  $g\Frob_\ell(g)^{-1}$ belongs to $N_{\calG}(\calT)(\ZZ_\ell^{un})$ and under the reduction map it is sent to $\bbar{g}\Frob_\ell(\bbar{g})^{-1}\in N_{\calG_{\FF_\ell}}(\calT_{\FF_\ell})(\FFbar_\ell)$.   The element of $W(\calG_{\FF_\ell},\calT_{\FF_\ell})$ represented by $\bbar{g}\Frob_\ell(\bbar{g})^{-1}$ belongs to the conjugacy class $\theta_{\calG_{\FF_\ell}}(\bbar{h})$ as in \S\ref{SS:tori over finite fields}.  Define $t:= g h g^{-1}$; it is an element of the set $\mathcal{I}_{v,\ell}$.  We have 
\[
\Frob_\ell(t)= \Frob_\ell(g) h \Frob_\ell(g)^{-1} = (g \Frob_\ell(g)^{-1})^{-1}\cdot t \cdot (g \Frob_\ell(g)^{-1})
\]
since $h$ is defined over $\QQ_\ell$.   Therefore, the conjugacy class of $\psi_{v,\ell}(\Frob_\ell)$ in $W(\calG_{\QQ_\ell}, \calT_{\QQ_\ell})=W(\G_{A,\ell},\T)$ is represented by $g\Frob_\ell(g)^{-1}$.  Since $g$ is defined over $\ZZ_\ell^{un}$, we deduce that $\psi_{v,\ell}$ is unramified at $\ell$.    With respect to our isomorphism $W(\G_{A,\ell},\T)=W(\calG_{\QQ_\ell},\calT_{\QQ_\ell}) \cong W(\calG_{\FF_\ell},\calT_{\FF_\ell})$, we find that $\psi_{v,\ell}(\Frob_\ell)$ lies in the conjugacy class $\theta_{\calG_{\FF_\ell}}(\bbar\rho_{A,\ell}(\Frob_v))$ of $W(\G_{A,\ell},\T)$.

Let $U_\ell$ be the set of $h\in \bbar\rho_{A,\ell}(\Gal_K)$ that are semisimple and regular in $\calG_{\FF_\ell}$ and satisfy $\theta_{\calG_{\FF_\ell}}(h)\subseteq C$; it is stable under conjugation by $\bbar\rho_{A,\ell}(\Gal_K)$.   If $v\in \calS_A$ satisfies $v\nmid \ell$ and $\bbar\rho_{A,\ell}(\Frob_v) \subseteq U_\ell$, then the above work shows that $\psi_{v,\ell}$ is unramified and $\psi_{v,\ell}(\Frob_\ell)$ lies in the conjugacy class $\theta_{\calG_{\FF_\ell}}(\bbar\rho_{A,\ell}(\Frob_v))\subseteq C$.

It remains to show that $U_\ell$ satisfies the second property in the statement of the lemma.   Proposition~\ref{P:image of Galois mod ell}(iii) tells us that $\bbar\rho_{A,\ell}(\Gal_{K'})$ contains the commutator subgroup of $\calG_{\FF_\ell}(\FF_\ell)$ for $\ell$ sufficiently large.   For such primes $\ell$, $\bbar{\rho}_{A,\ell}(\kappa)$ consists of cosets of the commutator subgroup of $\calG_{\FF_\ell}(\FF_\ell)$, and hence 
\[
\frac{|\bbar{\rho}_{A,\ell}(\kappa) \cap U_\ell|}{|\bbar{\rho}_{A,\ell}(\kappa)|} = \frac{|C|}{|W(\calG_{\FF_\ell},\calT_{\FF_\ell})|} + O(1/\ell)=\frac{|C|}{|W(\G_{A,\ell},\T)|} + O(1/\ell)
\]
by Lemma~\ref{L:JKZ estimate} where the implicit constant depends only on $A$ and $K'$ (the dimension of $\G_{A,\ell}$ is bounded in terms of $\dim A$, and hence there are only finite many possible Lie types for the groups $\G_{A,\ell}$ as $\ell$ varies).
\end{proof}

\section{Proofs of Theorem~\ref{T:main} and Theorem~\ref{T:Weil under MT}}
\label{S:proof main}

Fix an {absolutely simple} abelian variety $A$ defined over a number field $K$.  We have assumed that $K_A^\conn=K$; equivalently, all the groups $\G_{A,\ell}$ are connected.  Fix an embedding $K\subseteq\CC$ and let $\G_A\subseteq \GL_V$ be the Mumford-Tate group of $A$ where $V=H_1(A(\CC),\QQ)$.  Fix a maximal torus $\T$ of $\G_A$.  Let $\calS_A$ be the set of places from \S\ref{SS:calSA}.  We shall assume that the Mumford-Tate conjecture for $A$ holds starting in \S\ref{SS:Galois action}.

\subsection{Weights}
We first describe some properties of the representation $\G_A \hookrightarrow \GL_V$.   We will use the group theory of \cite{MR563476}*{\S3} (the results on \emph{strong Mumford-Tate pairs} in \cite{MR1603865}*{\S4} are also relevant).  

By Proposition~\ref{P:Faltings}(i), the commutant of $\G_A$ in $\End_{\QQ}(V)$ is naturally isomorphic to the ring $\Delta:=\End(A)\otimes_\ZZ \QQ$.  The ring $\Delta$ is a division algebra since $A$ is simple.  The center $E$ of $\Delta$ is a number field.   Define the integers $r:=[E:\QQ]$ and $m:=[\Delta:E]^{1/2}$.    The representation $\G_A \hookrightarrow \GL_V$ is irreducible since $\Delta$ is a division algebra.  

For each character $\alpha\in X(\T)$, let $V(\alpha)$ be the subspace of $V\otimes_\QQ \Qbar$ consisting of those vectors $v$ for which $t \cdot v = \alpha(t)v$ for all $t\in \T(\Qbar)$.    We say that $\alpha\in X(\T)$ is a \defi{weight} of $V$ if $V(\alpha)\neq 0$, and we denote the set of such weights by $\Omega$.  We have a decomposition $V\otimes_\QQ \Qbar =\oplus_{\alpha\in \Omega} V(\alpha)$, and hence 
\begin{equation} \label{E:weight decomp}
\det(xI - t) = \prod_{\alpha \in \Omega} (x-\alpha(t))^{m_\alpha}
\end{equation}
for each $t\in \T(\Qbar)$ where $m_\alpha:={\dim_{\Qbar} V(\alpha)}$ is the \defi{multiplicity} of $\alpha$.  The set $\Omega$ of weights is stable under the actions of $W(\G_A,\T)$ and $\Gal_\QQ$ on $X(\T)$, so $\Pi(\G_A,\T)$ also acts on $\Omega$. 

\begin{lemma} \label{L:Serre strong}
\begin{romanenum}
\item 
The representation $V \otimes_\QQ \Qbar$ of $\G_{A,\Qbar}$ is the direct sum of $r$ irreducible representations $V_1,\ldots, V_r$.  Let $\Omega_i\subseteq X(\T)$ be the set of weights of $V_i$.   Then $\Omega$ is the disjoint union of the sets $\Omega_1,\ldots,\Omega_r$.
\item 
The group $W(\G_A,\T)$ acts transitively on each set $\Omega_i$.   In particular, the action of $W(\G_A,\T)$ on $\Omega$ has $r$ orbits. 
\item \label{I:Serre strong ????}
The group $\Pi(\G_A,\T)$ acts transitively on $\Omega$. 
\item  \label{I:Serre strong, m}
For each $\alpha\in \Omega$, we have $m_\alpha=m$.
\end{romanenum}
\end{lemma}
\begin{proof} 
These properties all follow from the results of Serre in {\S3.2} (and in particular p.183) of \cite{MR563476}; note that the Mumford-Tate group $\G_A$ satisfies the hypotheses of that section.   Fix a Borel subgroup $\mathbf{B}$ of $\G_{A,\Qbar}$ that contains $\T$.  Serre shows that $\Omega= W(\G_A,\T)\cdot \Omega^+$ where $\Omega^+$ is the set of highest weights of the irreducible representations of $V\otimes_\QQ \Qbar$ (the notion of highest weight will depend on our choice of $\mathbf{B}$).    The set $\Omega^+$ has $r$ elements.   The sets $\Omega_1,\ldots, \Omega_r$ in the statement of the lemma are the orbits $W(\G_A,\T)\cdot \alpha$ with $\alpha\in \Omega^+$.  The group $\Gal_\QQ$ acts transitively on $\Omega_+$, so we find that $\Pi(\G_A,\T)$ acts transitively on $\Omega$.   That $\Pi(\G_A,\T)$ acts transitively on $\Omega$ implies that each weight $\alpha\in \Omega$ has the same multiplicity; Serre shows that it is $m$.
\end{proof}

We now give some basic arithmetic consequences of these geometric properties.   

\begin{lemma} \label{L:gamma details}
Let $L$ be an algebraically closed extension of $\Qbar$.  Fix a place $v\in\calS_A$ and an element $t\in \T(L)$ that satisfies $\det(xI-t)=P_{A_v}(x)$.  Then the map 
\[
\gamma\colon X(\T)\to\Phi_{A_v},\quad \alpha\mapsto \alpha(t)
\] 
is a well-defined homomorphism that satisfies $\gamma(\Omega)=\calW_{A_v}$.   The homomorphism $\gamma$ is surjective; it is an isomorphism if and only if the Mumford-Tate conjecture for $A$ holds.  
\end{lemma}
\begin{proof}
The map $\alpha\mapsto \alpha(t)$ certainly gives a homomorphism $\gamma\colon X(\T)\to L^\times$.   We need to show that $\gamma$ has image in $\Phi_{A_v}$.    By (\ref{E:weight decomp}), the roots of $\det(xI-t)$ in $L$ are the values $\alpha(t)$ with $\alpha \in \Omega$.     Since $P_{A_v}(x)=\det(xI-t)$ by assumption, we have $\calW_{A_v}=\{\alpha(t):\alpha\in \Omega\}=\gamma(\Omega)$.    The set $\Omega$ generates $X(\T)$ since $\G_A$ acts faithfully on $V$.  Since $\gamma(\Omega)=\calW_{A_v}$ and $\calW_{A_v}$ generates $\Phi_{A_v}$, we deduce that $\gamma(X(\T))=\Phi_{A_v}$.  This proves that $\gamma\colon X(\T)\to\Phi_{A_v}$ is a well-defined surjective homomorphism.

  The group $\Phi_{A_v}$ is a free abelian group of rank $\tilde r$ by our definition of $\calS_A$ where $\tilde r$ is the common rank of the groups $\G_{A,\ell}$.   The group $X(\T)$ is a free abelian group whose rank equals the rank of $\G_A$.  Since $\gamma$ is a surjective map of free abelian groups, we find that $\gamma$ is an isomorphism if and only if $\tilde r$ equals the rank of $\G_A$.  By \cite{MR1339927}*{Theorem~4.3}, the Mumford-Tate conjecture for $A$ holds if and only if $\tilde r$ equals the rank of $\G_A$.
\end{proof}

Using that $\calS_A$ has density 1, Theorem~\ref{T:main}(i) will follow immediately from the next lemma.

\begin{lemma} \label{L:reduce to poly}
Fix a place $v\in \calS_A$.   
\begin{romanenum}
\item
The abelian variety $A_v$ is isogenous to $B^{m}$ for an abelian variety $B$ over $\FF_v$.
\item  
If the Mumford-Tate conjecture for $A$ holds, then $P_{B}(x)$ is separable where $B/\FF_v$ is as in (i).
\item
If $P_{B}(x)$ is irreducible, then the abelian variety $B/\FF_v$ in (i) is absolutely simple.
\end{romanenum}
\end{lemma}
\begin{proof}
Fix a prime $\ell$ such that $v\nmid \ell$ and choose an embedding $\Qbar\hookrightarrow \Qbar_\ell$.   By Proposition~\ref{P:MT inclusion}, $\rho_{A,\ell}(\Frob_v)$ gives a conjugacy class in $\G_A(\QQ_\ell)$.  Choose an element $t\in \T(\Qbar_\ell)$ that is conjugate to $\rho_{A,\ell}(\Frob_v)$ in $\G_A(\Qbar_\ell)$.  By (\ref{E:weight decomp}) and Lemma~\ref{L:Serre strong}(\ref{I:Serre strong, m}), we have
\begin{equation} \label{E:weight factorization}
P_{A_v}(x)=\det(xI-t)=\Big(\prod_{\alpha\in \Omega}(x-\alpha(t))\Big)^m
\end{equation}
and hence $P_{A_v}(x)$ is the $m$-th power of a monic polynomial $Q(x)$ in $\ZZ[x]$.  

The field $\FF_v$ has prime cardinality $p:=N(v)$ since $v\in \calS_A$.   The polynomial $x^2-p$ does not divide $P_{A_v}(x)$; if it did, then $-1=(-\sqrt{p})/\sqrt{p}$ would belong to $\Phi_{A_v}$, which is impossible since $\Phi_{A_v}$ is torsion-free by our choice of $\calS_A$.    Theorem~\ref{L:Honda-Tate} thus implies that $A_v$ is isogenous to $B^m$ for some abelian variety $B/\FF_v$ which satisfies $P_{B}(x)=Q(x)$.  This proves (i).   If the Mumford-Tate conjecture for $A$ holds, then $Q(x)$ is separable by (\ref{E:weight factorization}) and Lemma~\ref{L:gamma details}; this proves (ii).

Finally we consider (iii); suppose that $P_B(x)$ is irreducible.  Take any positive integer $i$ and let $\FF$ be the degree $i$ extension of $\FF_v$.    We have $P_{B_{\FF}}(x) = \prod_{\pi \in \calW_{A_v}}(x-\pi^i)$ 
since $P_{B}(x)$ is separable with roots $\calW_{A_v}$.   For $\sigma\in \Gal_\QQ$ and $\pi_1,\pi_2\in \calW_{A_v}$, we claim that $\sigma(\pi_1^i)=\pi_2^i$ if and only if $\sigma(\pi_1)=\pi_2$.    If $\sigma(\pi_1)=\pi_2$, then we have $\sigma(\pi_1^i)=\pi_2^i$ by taking $i$-th powers.   If $\sigma(\pi_1^i)=\pi_2^i$, then $\sigma(\pi_1)/\pi_2$ equals 1 since it is an $i$-th root of unity that belongs to the torsion-free subgroup $\Phi_{A_v}$ of $\Qbar^\times$.  The group $\Gal_{\QQ}$ acts transitively on $\calW_{A_v}$ since $P_B(x)$ is irreducible.  The claim then implies that $P_{B_{\FF}}(x)\in \ZZ[x]$ is irreducible and hence $B_{\FF}$ is simple.   The abelian variety $B$ is absolutely simple since $i$ was arbitrary.
\end{proof}

\subsection{Galois action}    \label{SS:Galois action}
  For the rest of \S\ref{S:proof main}, we shall assume that the Mumford-Tate conjecture for $A$ holds.  Fix a place $v\in \calS_A$.  Choose an element $t_v \in \T(\Qbar)$ such that $\cl_{\G_A}'(t_v)=F_v'$ where $F_v' \in \Conj'(\G_A)(\QQ)$ is as in Theorem~\ref{T:Noot}; the place $v$ satisfies the condition of the theorem since $\Phi_{A_v}$ is torsion-free.   Since $F_v'=\cl_{\G_A}'(\rho_{A,\ell}(\Frob_v))$ for any prime $\ell$ satisfying $v\nmid \ell$, we may further assume that $t_v$ is chosen so that $\det(xI-t_v)=P_{A_v}(x)$.   Let $\Gamma$ be the subgroup of $\Aut(\T_{\Qbar})\cong \Aut(X(\T))$ from \S\ref{S:Noot}.

By Lemma~\ref{L:gamma details}, the map
\[
\gamma\colon X(\T) \to \Phi_{A_v},\quad \alpha\mapsto \alpha(t_v)
\]
is a homomorphism that satisfies $\gamma(\Omega)=\calW_{A_v}$; it is an isomorphism since we have assumed that the Mumford-Tate conjecture for $A$ holds.  For each $\sigma\in \Gal_\QQ$, we define $\psi_v(\sigma)$ to be the unique automorphism of $X(\T)$ for which the following diagram commutes:
\[
\xymatrix{ 
X(\T) \ar[r]^{\substack{\gamma\\\sim}} \ar[d]_{\psi_v(\sigma)} & \Phi_{A_v} \ar[d]^{\sigma}\\
X(\T) \ar[r]^{\substack{\gamma\\\sim}}  & \Phi_{A_v}.
}
\]
This defines a Galois representation
\[
\psi_v \colon \Gal_\QQ \to \Aut(X(\T)).
\]

For each prime $\ell$, we choose an embedding $\Qbar\hookrightarrow \Qbar_\ell$.   With respect to this embedding, the restriction map gives an injective homomorphism $\Gal_{\QQ_\ell}\hookrightarrow \Gal_\QQ$.    This embedding and the our assumption that the Mumford-Tate conjecture for $A$ holds, gives an isomorphism $W(\G_A,\T)\xrightarrow{\sim}W(\G_{A,\QQ_\ell}, \T_{\QQ_\ell})=W(\G_{A,\ell},\T_{\QQ_\ell})$.    If $\T_{\QQ_\ell}$ is split and $v\nmid \ell$, then using this isomorphism of Weyl groups and the construction of \S\ref{S:local Galois}, we have a group homomorphism 
\[
\psi_{v,\ell}\colon \Gal_{\QQ_\ell} \to W(\G_A,\T).
\]

\begin{lemma}  \label{L:Weyl class of psi}
Fix notation as above and let $\ell$ be a prime for with $\T_{\QQ_\ell}$ is split and $v\nmid \ell$.   Then for all $\sigma\in \Gal_{\QQ_\ell}$,  $\psi_v(\sigma)$ and $\psi_{v,\ell}(\sigma)$ are elements of $W(\G_A,\T)$ that lie in the same conjugacy class of $\Gamma$.
\end{lemma}
\begin{proof}
Recall that to define $\psi_{v,\ell}$, we chose an element $t_{v,\ell} \in \T(\Qbar_\ell)$ such that $t_{v,\ell}$ is conjugate to $\rho_{A,\ell}(\Frob_v)$ in $\G_{A,\ell}(\Qbar_\ell)=\G_A(\Qbar_\ell)$.   This implies that $\cl_{\G_A}'(t_{v,\ell})=\cl_{\G_A}'(\rho_{A,\ell}(\Frob_v))=F_v'$.   So there is a unique $\beta \in \Gamma$ such that $t_{v,\ell}=\beta(t_v)$.    Now take any $\sigma\in \Gal_{\QQ_\ell}$ and $\alpha\in X(\T)$.    We have 
\[
\sigma(\alpha(t_v))= \sigma( \alpha(\beta^{-1}(t_{v,\ell})))= \alpha(\beta^{-1}(\sigma(t_{v,\ell})))
\]
where we have used that $\beta$ and $\alpha$ are defined over $\QQ_\ell$ since $\T_{\QQ_\ell}$ is split.   By the definition of $\psi_{v,\ell}$, we have
\[
\sigma(\alpha(t_v))=\alpha(\beta^{-1}(\psi_{v,\ell}(\sigma)^{-1}(t_{v,\ell}))) = \big( \alpha \circ (\beta^{-1} \circ \psi_{v,\ell}(\sigma)^{-1} \circ \beta)\big) (t_v).
\]
From our characterization of $\psi_{v}(\sigma)$, we deduce that $\psi_v(\sigma)$ equals $\beta^{-1} \circ \psi_{v,\ell}(\sigma) \circ \beta$; it is an element of $W(\G_A,\T)$ since $\psi_{v,\ell}(\sigma)\in W(\G_A,\T)$ and $W(\G_A,\T)$ is a normal subgroup of $\Gamma$.
\end{proof}

Recall that we defined $k_{\G_A}$ to be the intersection of all the subfields $L\subseteq \Qbar$ for which $\G_{A,L}$ is split; it is a finite Galois extension of $\QQ$.   The following gives a strong constraint on the image of $\psi_v$.

\begin{lemma} \label{L:in the Weyl group}
With notation as above, $\psi_v(\Gal_{k_{\G_A}})$ is a subgroup of $W(\G_A,\T)$.
\end{lemma}
\begin{proof}
Let $L\subseteq \Qbar$ be a finite extension of $\QQ$ for which $\T_L$ is split.  Let $\Lambda$ be the set of primes $\ell$ for which $\psi_v$ is unramified at $\ell$, $v\nmid \ell$, and $\ell$ splits completely in $L$.   The torus $\T_{\QQ_\ell}$ is split for all $\ell\in \Lambda$.  From Lemma~\ref{L:Weyl class of psi}, we find that $\psi_v(\Frob_\ell)$ belongs to $W(\G_A,\T)$ for all $\ell\in \Lambda$.   The Chebotarev density theorem then ensures us that $\psi_v(\Gal_L)\subseteq W(\G_A,\T)$.

Now suppose that $L\subseteq \Qbar$ is \emph{any} finite extension of $\QQ$ for which $\G_{A,L}$ is split.  Choose a maximal torus $\T'$ of $\G_A$ for which $\T'_{L}$ is split.  Fix an element $g\in \G_{A}(\Qbar)$ such that $\T'_{\Qbar}=g \T_{\Qbar} g^{-1}$, and define $t'_v:=gt_v g^{-1}$.  We have $\cl_{\G_A}'(t'_v)=F_v'$ and $\det(xI-t_v')=P_{A_v}(x)$.    As above, we can define a homomorphism $\psi_v' \colon \Gal_\QQ\to \Aut(X(\T'))$ that is characterized by the property $\sigma(\alpha(t_v')) = \big(\psi_v'(\sigma) \alpha\big)(t_v')$ for all $\alpha\in X(\T')$ and $\sigma\in \Gal_\QQ$.   The argument from the beginning of the proof shows that $\psi_{v}'(\Gal_{L})\subseteq W(\G_A,\T')$.       We now need to relate $\psi_v$ and $\psi_v'$.

Define the isomorphisms $\beta \colon\T_{\Qbar}\to \T'_\Qbar,\, t\mapsto gtg^{-1}$ and $\beta_*\colon \Aut(\T_\Qbar)\to \Aut(\T'_\Qbar),\,  f\mapsto \beta\circ f \circ \beta^{-1}$.  One readily checks that $\beta_*(W(\G_A,\T))=W(\G_A,\T')$.   Take any $\alpha\in X(\T)$ and $\sigma\in \Gal_L$.  For the rest of the proof, it will be convenient to view $\psi_v(\sigma)$ and $\psi_v'(\sigma)$ as elements of $\Aut(\T_\Qbar)$ and $\Aut(\T_\Qbar')$, respectively.   By the defining property of $\psi_v'(\sigma)$, we have
\[
\sigma(\alpha(t_v))=\sigma((\alpha\circ \beta^{-1})(t_v')) =(\alpha\circ \beta^{-1} \circ \psi_v'(\sigma)^{-1})(t_v')= (\alpha\circ \beta^{-1} \circ \psi_v'(\sigma)^{-1}\circ \beta)(t_v)
\]
By our characterization of $\psi_v(\sigma)$, we deduce that $\psi_v(\sigma)= \beta^{-1} \circ \psi'_v(\sigma) \circ \beta=\beta_*^{-1}(\psi'_v(\sigma))$.    Therefore, $\psi_v(\Gal_{L})\subseteq \beta_*^{-1}(W(\G_A,\T'))=W(\G_A,\T)$.

We have shown that $\psi_v(\Gal_L)\subseteq W(\G_A,\T)$ for every finite extension $L/\QQ$ for which $\G_{A,L}$ is split.   It is then easy to show that $\psi_v(\Gal_{k_{\G_A}}) \subseteq W(\G_A,\T)$.
\end{proof}

We will now prove that $\psi_v$ has large image for most places $v$.

\begin{prop}  \label{P:sieving}
Fix a finite extension $L$ of $k_{\G_A}$.   Then $\psi_v(\Gal_L) = W(\G_A,\T)$ for all places $v\in \calS_A$ away from a set of density 0.
\end{prop}
\begin{proof}
By Lemma~\ref{L:in the Weyl group}, we know that $\psi_v(\Gal_L)$ is a subgroup of $W(\G_A,\T)$ for all $v\in \calS_A$.   There is no harm in replacing $L$ by a larger extension, so we may assume that $\T_L$ is split.  Let $\Lambda$ be the set of primes $\ell$ that split completely in $L$, and let $\Lambda_Q$ be the set of $\ell\in \Lambda$ that satisfy $\ell\leq Q$.   The torus $\T_{\QQ_\ell}$ is split for all $\ell\in \Lambda$.   After removing a finite number of primes from $\Lambda$, we may assume by Proposition~\ref{P:image of Galois mod ell}(i) that $\calG_{A,\ell}$ is a reductive scheme over $\ZZ_\ell$ for all $\ell\in\Lambda$.

Let $\mathscr{T}$ be the Zariski closure of $\T$ in the group scheme $\GL_{H_1(A(\CC),\ZZ)}$ over $\ZZ$; note that the generic fiber of $\GL_{H_1(A(\CC),\ZZ)}$ is $\GL_V$.  For $\ell$ sufficiently large, $\mathscr{T}_{\ZZ_\ell}$ is a torus over $\ZZ_\ell$.  Since the Mumford-Tate conjecture for $A$ has been assumed, we find that $\mathscr{T}_{\ZZ_\ell}$ is a maximal torus of $\calG_{A,\ell}$ for all sufficiently large primes $\ell$.   So after possibly removing a finite number of primes from $\Lambda$, we find that $\mathscr{T}_{\ZZ_\ell}$ is a split maximal torus of the reductive scheme $\calG_{A,\ell}$ for all $\ell\in \Lambda$.

Let $K'/K$ be an extension as in Proposition~\ref{P:independence}.  Fix a non-empty subset $C$ of $W(\G_A,\T)$ that is stable under conjugation by $\Gamma$.  For each $\ell\in\Lambda$, we can identify $C$ with a subset of $W(\G_{A,\ell},\T_{\QQ_\ell})=W((\calG_{A,\ell})_{\QQ_\ell},(\mathscr{T}_{\ZZ_\ell})_{\QQ_\ell})$.        With our fixed $C$ and $K'$, let $U_\ell$ be sets of Lemma~\ref{L:sieving input} for $\ell\in \Lambda$.

Let $\calV$ be the set of place $v\in \calS_A$ for which $\bbar\rho_{A,\ell}(\Frob_v)\not\subseteq U_\ell$ for all $\ell\in \Lambda$ that satisfy $v\nmid \ell$.   Let $\calV(Q)$ be the set of places $v\in \calS_A$ such that $\bbar\rho_{A,\ell}(\Frob_v)\not\subseteq U_\ell$ for all $\ell\in \Lambda_Q$ that satisfy $v\nmid \ell$.     By Proposition~\ref{P:independence density} and using that $\calS_A$ has density 1, we find that $\calV(Q)$ has density
\[
\delta_Q:=\sum_{\calC} \frac{|\calC|}{|\Gal(K'/K)|}\cdot \prod_{\ell\in \Lambda_Q} \frac{|\bbar\rho_{A,\ell}(\Gamma_\calC) \cap (\bbar\rho_{A,\ell}(\Gal_K)-U_\ell) |}{|\bbar\rho_{A,\ell}(\Gamma_\calC)|}
\]
where $\calC$ varies over the conjugacy classes of $\Gal(K'/K)$ and $\Gamma_\calC$ is the set of $\sigma\in \Gal_K$ for which $\sigma|_{K'}\in \calC$.  Using the bounds of Lemma~\ref{L:sieving input}, we have 
\[
\delta_Q \ll \prod_{\ell\in \Lambda_Q} \Big(1-\frac{|C|}{|W(\G_{A,\ell},\T_{\QQ_\ell})|} + O(1/\ell)\Big)=\prod_{\ell\in \Lambda_Q} \Big(1-\frac{|C|}{|W(\G_{A},\T)|} + O(1/\ell)\Big).
\] 
where the implicit constants do not depend on $Q$.  Since $\Lambda$ is infinite and $C$ is non-empty, we find that  $\lim_{Q\to+\infty}\delta_Q=0$.   Since $\calV$ is a subset of $\calV(Q)$ for every $Q$, we deduce that the density of $\calV$ exists and equals 0.  

Now take any place $v\in \calS_A-\calV$.  There is some prime $\ell\in \Lambda$ for which $v\nmid \ell$ and  $\bbar\rho_{A,\ell}(\Frob_v)\subseteq U_\ell$.  By the properties of $U_\ell$ from Lemma~\ref{L:sieving input}, we find that $\psi_{v,\ell}$ is unramified at $\ell$ and $\psi_{v,\ell}(\Frob_\ell)\subseteq C$.   Since $C$ is stable under conjugation by $\Gamma$, Lemma~\ref{L:Weyl class of psi} implies that $\psi_v(\Frob_\ell)\subseteq C$.  Since $\ell$ splits completely in $L$, we deduce that $\psi_v(\Gal_L)\cap C \neq \emptyset$.

For each $v\in \calS_A$, we have $\psi_v(\Gal_L)\subseteq W(\G_A,\T)$ by Lemma~\ref{L:in the Weyl group}.  By considering the finitely many $C$, we find that for all places $v\in \calS_A$ away from a set of density 0, we have $\psi_v(\Gal_L) \cap C \neq \emptyset$ for every non-empty subset $C$ of $W(\G_A,\T)$ that is stable under conjugation by $\Gamma$.  By Lemma~\ref{L:new Jordan}, we deduce that $\psi_v(\Gal_L)=W(\G_A,\T)$ for all places $v\in \calS_A$ away from a set of density 0.
\end{proof}

\subsection{Proof of Theorem~\ref{T:Weil under MT}}  
Take $v\in \calS_A$.  The group $\Phi_{A_v}$ is generated by $\calW_{A_v}$.  Using this and  Lemma~\ref{L:in the Weyl group}, we find that $\psi_v|_{\Gal_L}$ factors through an injective homomorphism 
$\Gal(L(\calW_{A_v})/L)\hookrightarrow W(\G_A,\T)\cong W(\G_A)$.   It is an isomorphism for all $v\in \calS_A$ away from a set of density 0 by Proposition~\ref{P:sieving}.   The theorem follows by noting that $\calS_A$ has density 1.

\subsection{Proof of Theorem~\ref{T:main}(ii)}
Fix a place $v\in\calS_A$.  The following lemma says that if the image of $\psi_v$ is as large as possible, then $P_{A_v}(x)$ factors in the desired manner.  Take any embedding $E\subseteq \Qbar$ and let $\widetilde{E}$ be the Galois closure of $E$ over $\QQ$.

\begin{lemma}   \label{L:big enough implies simple}
Let $L$ be a finite extension of $k_{\G_A}$ which contains $\widetilde{E}$.   If $\psi_{v}(\Gal_L)=W(\G_A,\T)$, then $P_{A_v}(x)$ is the $m$-th power of an irreducible polynomial.
\end{lemma}
\begin{proof}
The isomorphism $\gamma\colon X(\T)\to \Phi_{A_v}$ of \S\ref{SS:Galois action} gives a bijection between $\Omega$ and $\calW_{A_v}$.   By Lemma~\ref{L:Serre strong}(ii), the action of $W(\G_A,\T)$ partitions $\Omega$ into $r$ orbits.   Since $\psi_v(\Gal_L)=W(\G_A,\T)$ by assumption, we deduce that the $\Gal_L$-action partitions $\calW_{A_v}$ into $r$ orbits; equivalently, $P_{A_v}(x)$ has $r$ distinct irreducible factors in $L[x]$.     From Lemma~\ref{L:reduce to poly}, we know that $P_{A_v}(x)$ is the $m$-th power of a separable polynomial.  So there are distinct monic irreducible polynomials $Q_1(x),\ldots, Q_r(x) \in L[x]$ such that
\begin{equation} \label{E:factorization early}
P_{A_v}(x)= Q_1(x)^m\cdots Q_r(x)^m.
\end{equation}
We will identify these $r$ irreducible factors, but we first recall some basic facts about $\lambda$-adic representations where $\lambda$ is a finite place of $E$.  A good exposition on $\lambda$-adic representations can be found in \cite{MR0457455}*{I-II}.  

The ring $\End(A)\otimes_\ZZ\QQ$, and hence also the field $E$, acts on $V=H_1(A(\CC),\QQ)$.   Therefore, $V_\ell(A)=V\otimes_{\QQ}\QQ_\ell$ is a module over $E_\ell:= E\otimes_\QQ \QQ_\ell$.   We have $E_\ell = \prod_{\lambda |\ell} E_\lambda$, where $\lambda$ runs over the the places $\lambda$ of $E$ dividing $\ell$.	Setting $V_\lambda(A):= V_\ell(A)\otimes_{E_\ell} E_\lambda$, we have a decomposition $V_\ell(A)= \bigoplus_{\lambda|\ell} V_\lambda(A)$.  Since $E\subseteq \End(A)\otimes_\ZZ \QQ $, the action of $\Gal_K$ on $V_\ell(A)$ is $E_\ell$-linear.  Therefore, $\Gal_K$ acts $E_\lambda$-linearly on $V_\lambda(A)$ and hence defines a Galois representation
\[
\rho_{A,\lambda} \colon \Gal_K \to \Aut_{E_\lambda}(V_\lambda(A)).
\]
(Of course when $E=\QQ$, we have our usual $\ell$-adic representations.)    For each $\lambda$, we will denote the rational prime it divides by $\ell(\lambda)$.    Since $A$ has good reduction at $v$, there is a polynomial $P_{A_v,E}(x)\in E[x]$ such that 
\[
P_{A_v,E}(x) = \det(xI - \rho_{A,\lambda}(\Frob_v))
\]
for all finite places $\lambda$ of $E$ for which $v\nmid \ell(\lambda)$.  The connection with our polynomial $P_{A_v}(x)\in \QQ[x]$ is that
\[
P_{A_v}(x) = N_{E/\QQ}(P_{A_v,E}(x)),
\]
cf.~\cite{MR0222048}*{11.8--11.10}.

Choose a prime $\ell$ that splits completely in $E$ for which $v\nmid \ell$.  We then have a decomposition $V_\ell(A)=\prod_{\lambda|\ell} V_\lambda(A)$ and $\Gal_K$ acts on each of the $r=[E:\QQ]$ vector spaces $V_\lambda(A)$.    This implies that $V_\lambda(A)$ is a representation of $\G_{A,\ell}=\G_{A,\QQ_\ell}$ for each $\lambda$ dividing $\ell$.

Using Lemma~\ref{L:Serre strong}, we deduce that $V_\lambda(A)\subseteq V\otimes_\QQ \QQ_\ell$ is an absolutely irreducible representation of $\G_{A,\QQ_\ell}$ and that each weight has multiplicity $m$.   Therefore, $P_{A_v,E}(x)$ is the $m$-th power of a unique monic polynomial $Q_v(x)\in E[x]$, and hence $P_{A_v}(x)=N_{E/\QQ}(P_{A_v,E}(x))=N_{E/\QQ}(Q_v(x))^m$.  So
\[
P_{A_v}(x)= \prod_{\sigma\colon E\hookrightarrow L} \sigma(Q_v(x))^m
\]
where the product is over the $r$ field embeddings of $E$ into $L$ (this uses our assumption that $\widetilde{E}\subseteq L$).  From our factorization (\ref{E:factorization early}), we deduce that the polynomials $\sigma(Q_v(x))$ are irreducible over $L[x]$.  In particular, $Q_v(x)$ is irreducible over $E$.     That $Q_v(x)$ is irreducible in $E[x]$ implies that $P_{A_v}(x) =N_{E/\QQ}(Q_v(x))^m$ is a power of some irreducible polynomial over $\QQ$.   Since $P_{A_v}(x)\in \ZZ[x]$ is the $m$-th power of a separable polynomial, we deduce that $P_{A_v}(x)$ is the $m$-th power of an irreducible polynomial.
\end{proof}

Fix a finite extension $L$ of $k_{\G_A}$ which contains $\widetilde{E}$.   By Proposition~\ref{P:sieving}, there is a subset $\calT\subseteq \Sigma_K$ with density 0 such that $\psi_v(\Gal_L)=W(\G_A,\T)$ for all $v\in \calS_A-\calT$.  By Lemma~\ref{L:big enough implies simple} and Lemma~\ref{L:reduce to poly}, we deduce that for all $v\in\calS_A-\calT$, $A_v$ is isogenous to $B^m$ where $B$ is an absolutely simple abelian variety over $\FF_v$.   Our theorem follows by noting that $\calS_A$ has density 1 and $\calT$ has density 0.

\section{Proof of Theorem~\ref{T:general}} \label{S:general proof}

After replacing $A$ by an isogenous abelian variety, we may assume that $A=\prod_{i=1}^s A_i^{n_i}$ where the $A_i$ are simple abelian varieties over $K$ which are pairwise non-isogenous.    

\begin{lemma} \label{L:factors are OK}
Fix an integer $1\leq i\leq s$.   The abelian variety $A_i/K$ is absolutely simple, $K_{A_i}^\conn=K$, and the Mumford-Tate conjecture for $A_i$ holds.
\end{lemma}
\begin{proof}
Take any prime $\ell$.  Viewing $A_i$ as one of the factors of $A$, we may view $V_\ell(A_i)$ as a subspace of $V_\ell(A)$ stable under $\Gal_K$.   Restriction to $V_\ell(A_i)$ defines a homomorphism $\pi\colon \G_{A,\ell}\to \G_{A_i,\ell}$ of algebraic groups for which $\pi(\G_{A,\ell})$ is Zariski dense in $\G_{A_i,\ell}$.   Since $\G_{A,\ell}$ is connected by our assumption $K_A^\conn=K$, we deduce that $\G_{A_i,\ell}$ is also connected.  Therefore, $K_{A_i}^\conn=K$.

Since $A_i$ is simple, we know by Faltings that $V_\ell(A_i)$ is an irreducible $\QQ_\ell[\Gal_K]$-module, and is hence an irreducible representation of $\G_{A_i,\ell}$.   Take any finite extension $L$ of $K$.   The group $\rho_{A_i,\ell}(\Gal_L)$ is Zariski dense in $\G_{A_i,\ell}$ since it is connected.   So $V_\ell(A_i)$ is an irreducible $\QQ_\ell[\Gal_L]$-module, and hence $A_{i,L}$ is simple.    Since $L$ was an arbitrary finite extension of $K$, we deduce that $A_i$ is absolutely simple.

Again by viewing $A_i$ as one of the factors of $A$, we can view $H_1(A_i(\CC),\QQ)$ as a subspace of $H_1(A(\CC),\QQ)$, which induces a homomorphism $\G_{A}\to \G_{A_i}$.   One can show that this is compatible with the corresponding map $\pi$, and that the Mumford-Tate conjecture for $A_i$ follows from our assumption that the Mumford-Tate conjecture for $A$ holds.
\end{proof}

Fix an integer $1\leq i\leq s$.  Lemma~\ref{L:factors are OK} allows us to apply Theorem~\ref{T:main}(ii) to each $A_i$.  By Theorem~\ref{T:main}(ii), there is a subset $\calT_i\subseteq \Sigma_K$ with density 0 such that for all $v\in \Sigma_K-\calT_i$, $A_i$ modulo $v$ is isogenous to $ B_{i,v}^{m_i }$ where $B_{i,v}$ is an absolutely simple abelian variety over $\FF_v$.   Set $\calT=\bigcup_{i=1}^s \calT_i$; it has density 0.   For all $v\in \Sigma_K-\calT$, we find that $A_v$ is isogenous to a product $\prod_{i=1}^s B_{i,v}^{m_in_i}$ where each $B_{i,v}$ is absolutely simple over $\FF_v$.

It remains to show that the abelian varieties $B_{1,v},\ldots, B_{s,v}$ are pairwise non-isogenous for all $v\in \Sigma_k-\calT$ away from a set of density 0.   It suffices to consider fixed $1\leq i<j \leq s$.  Fix a prime $\ell$.  If $B_{i,v}$ is isogenous to $B_{j,v}$, then  $A_i^{m_j}$ and $A_j^{m_i}$ modulo $v$ are isogenous, and hence
\[
m_j\tr(\rho_{A_i,\ell}(\Frob_v))=\tr(\rho_{{A_i}\!^{m_j},\ell}(\Frob_v))=\tr(\rho_{{A_j}\!^{m_i},\ell}(\Frob_v))=m_i\tr(\rho_{{A_j},\ell}(\Frob_v))
\] 
if $v\nmid\ell$.  Let $\mathcal{P}$ be the set of $v\in \Sigma_K$ for which $A_i$ and $A_j$ have good reduction and $m_j\tr(\rho_{A_i,\ell}(\Frob_v))=m_i\tr(\rho_{{A_j},\ell}(\Frob_v))$.  To finish the proof, it suffices to show that $\mathcal{P}$ has density 0.

 We can view $\G_{A_i\times A_j,\ell}$ as an algebraic subgroup of $\G_{A_i,\ell}\times \G_{A_j,\ell}$.   Let $W/\QQ_\ell$ be the subvariety of $\G_{A_i,\ell}\times \G_{A_j,\ell}$ defined by the equation $m_j \tr(g) = m_i \tr(g')$ with $(g,g')\in \G_{A_i,\ell}\times \G_{A_j,\ell}$.    
 
First suppose that $\G_{A_i\times A_j,\ell} \subseteq W$.  Then $\tr\circ \rho_{{A_i}\!^{m_j},\ell}=m_j\cdot \tr\circ \rho_{A_i,\ell} =m_i\cdot \tr\circ\rho_{{A_j},\ell} =\tr\circ \rho_{{A_j}\!^{m_i},\ell}$, and hence $A_i^{m_j}$ and $A_j^{m_i}$ are isogenous by the work of Faltings.    Since $A_i$ and $A_j$ are simple, we deduce that they are isogenous; this contradicts our factorization of $A$.

 Therefore, $\G_{A_i\times A_j,\ell} \not \subseteq W$.   Arguing as in Lemma~\ref{L:factors are OK}, we find that the group $\G_{A_i\times A_j,\ell}$ is connected.   Since $\G_{A_i\times A_j,\ell}$ is connected,  $\G_{A_i\times A_j,\ell}\cap W$ is of codimension at least 1 in $\G_{A_i\times A_j,\ell}$.   The Chebotarev density theorem then implies that $\mathcal{P}$ has density 0, as desired.   
  
\section{Remarks on Conjecture~\ref{C:main}} \label{S:conj ends}

We restate Conjecture~\ref{C:main}, but now emphasize that $\Pi(\G_A)$ is the group defined in \S\ref{SS:Weyl group}.

\begin{conj} \label{C:conj final}
Let $A$ be a non-zero abelian variety defined over a number field $K$ that satisfies $K_A^{\conn}=K$.   Then $\Gal(\QQ(\calW_{A_v})/\QQ)\cong \Pi(\G_A)$ for all $v\in \Sigma_K$ away from a subset with natural density 0.
\end{conj}

When $A$ is also absolutely simple, we shall show that this conjecture follows from other well-known conjectures which have already been discussed.

\begin{thm} \label{T:conj final case}
Let $A$ be an absolutely simple abelian variety defined over a number field $K$ that satisfies $K_A^{\conn}=K$.   Suppose that the Mumford-Tate conjecture for $A$ holds and that a class $F_v\in \Conj(\G_A)(\QQ)$ as in Conjecture~\ref{C:Frob conj} exists for all $v\in \Sigma_K$ away from a set of density 0.   Then Conjecture~\ref{C:conj final} for $A$ is true.
\end{thm}

\begin{remark} \label{R:Noot good}
\begin{romanenum}
\item
Let $A$ be an absolutely simple abelian variety defined over a number field $K$ such that $K_A^{\conn}=K$ and such that the Mumford-Tate conjecture for $A$ holds.  Suppose further that $(\G_A^\der)_\Qbar$ has no normal subgroups isomorphic to $\SO(2k)_\Qbar$ with $k\geq 4$.  Theorem~\ref{T:Noot} then implies that a class $F_v\in \Conj(\G_A)(\QQ)$ as in Conjecture~\ref{C:Frob conj} exists for all $v\in \calS_A$ (where $\calS_A$ is the set of density 1 from \S\ref{SS:calSA}).
\item
Theorem~\ref{T:conj final case} should remain true without the assumption that $A$ is absolutely simple.   We required this assumption in order to apply Proposition~\ref{P:sieving}.  That proposition in turn needed the assumption in order to use Proposition~\ref{P:image of Galois mod ell} (note that in \cite{MR1944805}*{\S3.4}, Wintenberger shows that the special fiber of $\calG_{A,\ell}$ agrees with the reductive group constructed by Serre, but only in the case where $A$ is absolutely simple).
\item
Under the stronger hypotheses of Theorem~\ref{T:conj final case} it is easier to prove Theorem~\ref{T:main}.   Using that $\Pi(\G_A,\T)$ acts transitively on the set of weights $\Omega$ (Lemma~\ref{L:Serre strong}), one can show that if $\Gal(\QQ(\calW_{A_v})/\QQ))\cong \Pi(\G_A)$, then $\Gal_\QQ$ acts transitively on the set $\calW_{A_v}$.   This avoids the more complicated argument in the proof of Lemma~\ref{L:big enough implies simple}.
\end{romanenum}
\end{remark}

\subsection{Example: abelian varieties of Mumford type} \label{SS:Mumford type}

There are abelian varieties $A/K$ of dimension 4 with $\End(A_{\Kbar})=\ZZ$ for which $\G_A \not\cong \GSp_{8,\QQ}$.   We say that such an abelian variety is of \defi{Mumford type}.  Such abelian varieties were show to exist by Mumford in \cite{MR0248146}.    For further details about such varieties, see \cite{MR1739726}.  

Let $A$ be an abelian variety over a number field $K$ that is of Mumford type and satisfies $K_A^{\conn}=K$.   Let $\G_A^\der$ be the derived subgroup of $\G_A$.   One can show that the group $\G_A^\der$ is simple over $\QQ$ and that $(\G_A^\der)_{\Qbar}$ is isogenous to $\SL_{2,\Qbar}^3$.   The Mumford-Tate conjecture for $A$ holds by Theorem~5.15 of \cite{MR1603865}.    By Theorem~\ref{T:main}, we deduce that the reduction $A_v/\FF_v$ is absolutely simple for all $v\in \Sigma_K$ away from a set of density 0; this is Theorem~C of \cite{Achter-effective}.    

By Theorem~\ref{T:conj final case} and Remark~\ref{R:Noot good}(i), we also deduce that $\Gal(\QQ(\calW_{A_v})/\QQ)$ is isomorphic to $\Pi(\G_A)$ for all $v\in \Sigma_K$ away from a set of density 0.

Finally, let us describe the possibilities for the group $\Pi(\G_A)$.  The center of $\G_A$ is the group of homotheties $\G_m$ since $\End(A_{\Kbar})=\ZZ$.  Since the center of $\G_A$ is split, we find that the groups $\Pi(\G_A)$ and $\Pi(\G_A^\der)$ are isomorphic.    Let $\Phi$ be a root system associated to $\G_A^\der$.   Using that $\G_A^\der$ is semisimple, we see that the group $\Pi(\G_A)$ is isomorphic to a subgroup of $\Aut(\Phi)$ which contains the Weyl group $W(\Phi)\cong W(\G_A)\cong (\ZZ/2\ZZ)^3$.   The group $\Aut(\Phi)$ is isomorphic to the semidirect product $(\ZZ/2\ZZ)^3 \rtimes S_3$, where $S_3$ acts on $(\ZZ/2\ZZ)^3$ by permuting coordinates.   Using that the algebraic group $\G_A^\der$ is simple over $\QQ$, we find that $\Pi(\G_A)$ must contain an element of order 3.   Therefore,  $\Pi(\G_A)$ is isomorphic to $(\ZZ/2\ZZ)^3 \rtimes S_3$ or $(\ZZ/2\ZZ)^3 \rtimes A_3$.

\begin{remark}
We have just shown that $P_{A_v}(x)$ is irreducible and $\Gal(\QQ(\calW_{A_v})/\QQ)\cong \Pi(\G_A)$ for all $v\in \Sigma_K$ away from a set of density 0.    Fix such a place $v$.   Even though $P_{A_v}(x)$ is irreducible, we find that $P_{A_v}(x) \pmod{\ell}$ is reducible in $\FF_\ell[x]$ for \emph{every} prime $\ell$ (one uses that the polynomial $P_{A_v}(x)$ has degree 8 while $\Pi(\G_A)$ has no elements of order 8).
\end{remark}

\subsection{Proof of Theorem~\ref{T:conj final case}}
Fix an embedding $K\subseteq \CC$ and let $\G_A$ be the Mumford-Tate group of $A$.   Choose a maximal torus $\T$ of $\G_A$.     By assumption, there is a set $\calS\subseteq \Sigma_K$ with density 1 such that an element $F_v\in \Conj(\G_A)(\QQ)$ as in Conjecture~\ref{C:Frob conj} exists for all $v\in \calS$.    Let $\calS_A$ be the density 1 subset of $\Sigma_K$ from \S\ref{SS:calSA}.   Without loss of generality, we may assume that $\calS\subseteq \calS_A$.

Take a place $v\in\calS$, and define the set
\[
\mathcal{J}_v := \big\{ t \in \T(\Qbar) : \cl_{\G_A}(t)=F_v \big\}.
\]
Choose an element $t_v\in\mathcal{J}_v$; it satisfies $\det(xI-t_v)=P_{A_v}(x)$.  By Lemma~\ref{L:gamma details}, the map $\gamma_v\colon X(\T)\to \Phi_{A_v},\, \alpha \mapsto \alpha(t_v)$ is an isomorphism of free abelian groups (this uses our assumption that the Mumford-Tate conjecture for $A$ holds).   There is thus a unique homomorphism $\psi_v\colon \Gal_\QQ \to \Aut(X(\T))$ such that $\sigma(\alpha(t_v))=\big(\psi_v(\sigma)\alpha\big)(t_v)$ for all $\sigma\in \Gal_\QQ$ and $\alpha\in X(\T)$.

We will now show that the image of $\psi_v$ lies in $\Pi(\G_A,\T)$.   Conjugation induces an action of $W(\G_A,\T)$ on $\mathcal{J}_v$; this action is simply transitive (the group $W(\G_A,\T)$ acts faithfully on $\T_\Qbar$ since the subgroup generated by each $t\in \mathcal{J}_v$ is Zariski dense in $\T_\Qbar$).   Since $F_v$ and $\cl_{\G_A}$ are defined over $\QQ$, the set $\mathcal{J}_v$ is also stable under the action of $\Gal_{\QQ}$.   So for each $\sigma\in \Gal_\QQ$, there is a unique $w_\sigma \in W(\G_A,\T)$ such that $\sigma(t_v)=w_\sigma^{-1}(t_v)$.   For $\alpha\in X(\T)$, we have
\[
\sigma(\alpha(t_v)) = \sigma(\alpha)(\sigma(t_v))=\sigma(\alpha)(w_\sigma^{-1}(t_v))= (\sigma(\alpha)\circ w_\sigma^{-1})(t_v).
\]
Therefore, 
 \begin{equation} \label{E:final psi formula}
 \psi_v(\sigma) \alpha = \sigma(\alpha) \circ w_\sigma
 \end{equation}
for all $\sigma\in\Gal_\QQ$ and $\alpha\in X(\T)$.  Since $w_\sigma \in W(\G_A,\T)$, we find that $\psi_v(\sigma)$ belongs to $\Pi(\G_A,\T)$ for all $\sigma\in \Gal_\QQ$.

Recall that $W(\G_A,\T)$ is a normal subgroup of $\Pi(\G_A,\T)$.   Define the homomorphism
\[
\bbar{\psi}_v\colon \Gal_\QQ \xrightarrow{\psi_v} \Pi(\G_A,\T) \twoheadrightarrow \Pi(\G_A,\T)/W(\G_A,\T).
\]
From (\ref{E:final psi formula}), we see that $\bbar{\psi}_v$ agrees with the composition of the homomorphism $\varphi_\T\colon \Gal_\QQ \to \Pi(\G_A,\T)$ from \S\ref{S:reductive} with the quotient map $\Pi(\G_A,\T)\to \Pi(\G_A,\T)/W(\G_A,\T)$.   In particular, we find that $\bbar{\psi}_v$ is surjective.   Therefore, we have $\psi_v(\Gal_\QQ)=\Pi(\G_A,\T)$ if and only if $\psi_v(\Gal_\QQ)\supseteq W(\G_A,\T)$.

Using Proposition~\ref{P:sieving}, we deduce that $\psi_v(\Gal_\QQ)=\Pi(\G_A,\T)$ for all $v\in \calS$ away from a set of density 0 (in fact, it would be much easier to prove Proposition~\ref{P:sieving} in the current setting since we do not have the extraneous group $\Gamma$ to deal with) .   Using that $\calW_{A_v}$ generates $\Phi_{A_v}$, we find that $\psi_v$ factors through an injective homomorphism $\Gal(\QQ(\calW_{A_v})/\QQ)\hookrightarrow \Pi(\G_A,\T)\cong \Pi(\G_A)$; the theorem follows immediately.

\section{Effective bounds} \label{S:large sieve}

For each place $v\in \Sigma_K$, we define $N(v)$ to be the cardinality of the field $\FF_v$.  For each subset $\scrS$ of $\Sigma_K$ and real number $x$, we define $\scrS(x)$ to be the set of $v\in \scrS$ that satisfy $N(v)\leq x$.   

Let $A$ be an absolutely simple abelian variety defined over a number field $K$ such that $K_A^\conn=K$.  Define the integer $m=[\End(A)\otimes_\ZZ\QQ:E]^{1/2}$ where $E$ is the center of the division algebra $\End(A)\otimes_\ZZ \QQ$.  Let $d$ and $r$ be the dimension and rank, respectively, of $\G_A$.  The following makes Theorem~\ref{T:main}(i) effective.

\begin{prop} \label{P:main effective}
Let $\scrS$ be the set of places $v\in \Sigma_K$ such that $A_v$ is \emph{not} isogenous to $B^m$ for some abelian variety $B$ over $\FF_v$.  Then $|\scrS(x)| \ll \frac{x}{(\log x)^{1+1/d}}  \cdot ((\log\log x)^{2} \log\log\log x)^{1/d}$.  If the Generalized Riemann Hypothesis (GRH) is true, then $|\scrS(x)| \ll x^{1-\frac{1}{2d}} (\log x)^{-1+2/d}$.
\end{prop}

We can also state an effective version of Theorem~\ref{T:main}(ii) and Theorem~\ref{T:Weil under MT}.

\begin{thm} \label{T:main effective}
Suppose that the representations $\{\rho_{A,\ell}\}_\ell$ are \emph{independent}, i.e., $\big(\prod_{\ell}\rho_{A,\ell}\big)(\Gal_K)=\prod_\ell \rho_{A,\ell}(\Gal_K)$ (by Proposition~\ref{P:independence} this can be achieved by replacing $K$ with a finite extension).   Assume that the Mumford-Tate conjecture for $A$ holds.

Let $\scrS_1$ be the set of places $v\in \Sigma_K$ for which $A_v$ is \emph{not} isogenous to $B^m$ for some absolutely simple abelian variety $B/\FF_v$.  Fix a finite extension $L$ of $k_{\G_A}$ and let $\scrS_2$ be the set of places $v\in \Sigma_K$ for which $\Gal(L(\calW_{A_v})/L)$ is \emph{not} isomorphic to $W(\G_A)$.  

Then 
\[
|\scrS_i(x)|\ll \frac{x(\log\log x)^{1+1/(3d)}}{(\log x)^{1+1/(6d)}}.
\]
If the Generalized Riemann Hypothesis (GRH) is true, then
\[
|\scrS_i(x)|\ll x^{1-\frac{1}{4d+2r+2}} (\log x)^{\frac{2}{2d+r+1}}.
\]
\end{thm}

These bounds will be an application of the large sieve as developed in \cite{large-sieve}. Cases where $\G_A\cong \GSp_{2\dim(A),\QQ}$ was handled in \cite{large-sieve}*{\S1.4} and many other cases were proved by Achter, see \cite{Achter-effective}*{Theorem~B}.  For comparison, note that  $|\Sigma_K(x)|\sim x/\log x$ as $x\to +\infty$.

\subsection{$\ell$-adic subvarieties}

\begin{lemma} \label{L:Serre ladic}
Fix a prime $\ell$ and a proper subvariety $V$ of $\G_{A,\ell}$ which is stable under conjugation.   Let $\scrS$ be the set of place $v\in \Sigma_K$ for which $A$ has good reduction, $v\nmid\ell$, and $\rho_{A,\ell}(\Frob_v)\subseteq V(\QQ_\ell)$.  Then 
\[
|\scrS(x)| \ll \frac{x}{(\log x)^{1+1/d}}  \cdot ((\log\log x)^{2} \log\log\log x)^{1/d}.
\]
If GRH holds, then $|\scrS(x)| \ll x^{1-\frac{1}{2d}} (\log x)^{-1+2/d}$.
\end{lemma}
\begin{proof}
Let $d'$ be the dimension of $\G_{A,\ell}$.  The variety $V$ has dimension at most $d'-1$ since it is a proper subvariety of the connected group $\G_{A,\ell}$.  By Proposition~\ref{P:Bogomolov}, $\rho_{A,\ell}(\Gal_K)$ has dimension $d'$ as an $\ell$-adic Lie group.   As an $\ell$-adic analytic variety, $V(\QQ_\ell)\cap \rho_{A,\ell}(\Gal_K)$ has dimension at most $d'-1$.  By \cite{MR644559}*{Th\'eor\`eme~10(i)}, we have
\[
|\scrS(x)|\ll \frac{x}{\log x}\Big( \frac{(\log\log x)^2 \log\log\log x}{\log x}\Big)^{1/d'} = \frac{x}{(\log x)^{1+1/d'}}  \cdot ((\log\log x)^{2} \log\log\log x)^{1/d'}
\]
Assuming GRH, \cite{MR644559}*{Th\'eor\`eme~10(ii)} implies that
\[
|\scrS(x)| \ll \frac{x}{\log x} \Big(\frac{(\log x)^2}{x^{1/2}}\Big)^{1/d'} = x^{1-\frac{1}{2d'}} (\log x)^{-1+2/d'}.
\]
We have $d'\leq d$ by Proposition~\ref{P:MT inclusion}; the lemma then quickly follows.
\end{proof}

We can now consider our set $\calS_A$ from \S\ref{SS:calSA}.

\begin{lemma} \label{L:calSA effective}
We have $|\Sigma_K(x)-\calS_A(x)|\ll \frac{x}{(\log x)^{1+1/d}}  \cdot ((\log\log x)^{2}\log\log\log x)^{1/d}$.  If GRH holds, then $|\Sigma_K(x)-\calS_A(x)| \ll x^{1-\frac{1}{2d}} (\log x)^{-1+2/d}$.
\end{lemma}
\begin{proof}
There are only finitely many places $v$ for which $A$ has bad reduction.   If $v\in \Sigma_K(x)$ satisfies $N(v)=p^e$ with $e>1$, then $p\leq \sqrt{x}$.  Using that at most $[K:\QQ]$ places of $K$ lie over a given prime $p$, we find that $|\{v\in \Sigma_K(x): N(v) \text{ not prime}\}| \leq [K:\QQ] \sqrt{x}$.  

Fix a prime $\ell$.  It thus suffices to consider the set $\scrS$ of places $v\in \Sigma_K$ for which $A$ has good reduction and $v\nmid \ell$ such that $\Phi_{A_v}$ is \emph{not} a free abelian group with rank equal to the common rank of the groups $\G_{A,\ell}$.  In \cite{MR1441234}*{\S2}, Pink and Larsen show that there is a proper subvariety $V$ of $\G_{A,\ell}$  stable under conjugation such that if $\rho_{A,\ell}(\Frob_v) \notin V(\QQ_\ell)$, then $v\in \scrS$.  [Let $\T_v$ be the algebraic subgroup of $\G_{A,\ell}$ generated by a representation of $\rho_{A,\ell}(\Frob_v)$.   Then $\Phi_{A_v}$ is a free abelian group with rank equal to the rank of $\G_{A,\ell}$ if and only if $\T_v$ is a maximal torus of $\G_{A,\ell}$ and $\rho_{A,\ell}(\Frob_v)$ is ``neat''.]

The required bounds for $|\scrS(x)|$ then follow from Lemma~\ref{L:Serre ladic}.
\end{proof}

\begin{proof}[Proof of Proposition~\ref{P:main effective}]
The proposition follows from Lemma~\ref{L:reduce to poly} and Lemma~\ref{L:calSA effective}.
\end{proof}

\subsection{Proof of Theorem~\ref{T:main effective}}
For a finite group $G$, we denote its set of conjugacy classes by $G^\sharp$.

\begin{lemma} \label{L:group bounds}
Let $\G$ be a split and connected reductive group defined over a finite field $\FF_q$.   Let $d$ and $r$ be the dimension and rank of $\G$, respectively. 
\begin{romanenum}
\item We have $|\G(\FF_q)|\leq q^d$.
\item There is a constant $\kappa\geq 1$, depending only on $d$ and $r$, such that $|\G(\FF_q)^\sharp| \leq \kappa q^r$.
\end{romanenum}
\end{lemma}
\begin{proof}
The reductive group $\G$ is the almost direct product of a split torus and a split semisimple group.   We can then reduce (i) and (ii) to the case where $\G$ is a split torus or $\G$ is a connected and split semisimple group.  If $\G$ is a split torus, then $r=d$ and we have $|\G(\FF_q)|=|\G(\FF_q)^\sharp|= (q-1)^d\leq q^d$.

Now suppose that $\G$ is semisimple.  We first prove (i).  The cardinality $|\G(\FF_q)|$ does not change under isogeny, so we may assume that $\G$ is simply connected.  The group $\G$ is then a product of simple, simply connected, and split semisimple groups (the number of factors being bounded in terms of $d$); so we may assume further that $\G$ is simple.   There are positive integers $a_i$ such that $|\G(\FF_q)|=q^d\prod_{i=1}^r\big(1-\frac{1}{q^{a_i}}\big)$; this can be deduced from \cite{MR0466335}*{Theorem~25(a)}.  Therefore, $|\G(\FF_q)|\leq q^d$.

We now prove (ii).   If $H$ is a proper subgroup of a finite group $G$, then we have inequalities $|H^\sharp| \leq [G:H] |G^\sharp|$  and $|G^\sharp| \leq [G:H] |H^\sharp|$, cf.~\cite{MR0125162}.   Using these inequalities, we can reduce part (ii) to showing that if $G$ is a finite simple group of Lie type over $\FF_q$ which arises from a simple algebraic group of rank $r$.  Then $|G^\sharp| \leq \kappa q^r$ for some constant $\kappa\geq 1$ depending only on $r$; this follows from \cite{MR1489911}*{Theorem~1}.
\end{proof}

\begin{prop} \label{P:large sieve}
Fix a set $\Lambda$ of rational primes with positive density such that $\calG_{A,\ell}$ is a split reductive group scheme over $\ZZ_\ell$ for all $\ell\in \Lambda$.  For each prime $\ell\in \Lambda$, fix a subset $U_\ell$ of $\bbar\rho_{A,\ell}(\Gal_K)$ that is stable under conjugation and satisfies $|U_\ell|/|\bbar\rho_{A,\ell}(\Gal_K)| = \delta + O(1/\ell)$ for some $0<\delta<1$, where $\delta$  and the implicit constant do not depend on $\ell$.  Let $\calV$ be the set of place $v\in \Sigma_K$ for which $A$ has good reduction and for which $\bbar\rho_{A,\ell}(\Frob_v)\not\subseteq U_\ell$ for all $\ell\in \Lambda$ that satisfy $v\nmid \ell$.   
\begin{alphenum}
\item Then
\[
|\calV(x)| \ll \frac{x(\log\log x)^{1+1/(3d)}}{(\log x)^{1+1/(6d)}}.
\]
\item If GRH holds, then
\[
|\calV(x)| \ll x^{1-\frac{1}{4d+2r+2}} (\log x)^{\frac{2}{2d+r+1}}.
\]
\end{alphenum}
\end{prop}
\begin{proof}
For each $\ell \in \Lambda$, we set $H_\ell:= \bbar\rho_{A,\ell}(\Gal_K)$.  Take any prime $\ell\in \Lambda$, and let $\G/\FF_\ell$ be the special fiber of $\calG_{A,\ell}$.  By Lemma~\ref{L:group bounds}(i), we have $|H_\ell| \leq |\G(\FF_\ell)|\leq \ell^d$.    We have an inequality $|H_\ell^\sharp|\leq [\G(\FF_\ell): H_\ell]\cdot |\G(\FF_\ell)^\sharp|$ (see the comments following \cite{MR0125162}*{Theorem~2}).  By Proposition~\ref{P:image of Galois mod ell}(ii) and Lemma~\ref{L:group bounds}(ii), there is a constant $\kappa\geq 1$ which does not depend on $\ell$ such that $|H_\ell^\sharp|\leq \kappa \ell^r$.

We now set some notation so that we may apply the sieve of \cite{large-sieve}.   After possibly removing a finite number of primes from $\Lambda$, there will be a constant $\delta'>0$ such that $|U_\ell|/|\bbar\rho_{A,\ell}(\Gal_K)| \geq \delta'$ for all $\ell\in \Lambda$.  Let $\Lambda_Q$ be the set of $\ell\in\Lambda$ that satisfy $\ell\leq Q$ and let $\mathcal{Z}(Q)$ be the set of subsets $D$ of $\Lambda_Q$ that satisfy $\prod_{\ell\in D} \kappa \ell \leq Q$.  Define the function
\[
L(Q):=\sum_{D\in \mathcal{Z}(Q)} \prod_{\ell\in D} \frac{\delta'}{1-\delta'}.
\]
For $Q$ large enough, we have $L(Q) \geq \sum_{\ell\in \Lambda, \ell\leq Q/\kappa} \delta'/(1-\delta') \gg Q/\log Q$ where the implicit constant does not depend on $Q$; this uses that $\Lambda$ has positive density.   For each $D \in \mathcal{Z}(Q)$, we define the group $H_D:= \prod_{\ell\in D} H_\ell$.   For $D\in \mathcal{Z}(Q)$, we have
\begin{equation} \label{E:H bounds}
|H_D|\leq \prod_{\ell\in D}\ell^d \leq Q^d \quad \text{ and } |H_D^\sharp| \leq \prod_{\ell\in D} \kappa\ell^r \leq \big(\prod_{\ell\in D} \kappa\ell\big)^r\leq Q^r.
\end{equation}

We first consider the unconditional case.  Theorem~3.3(i) of \cite{large-sieve} implies that for a sufficiently small positive constant $c$, we have
\[
|\calV(x)| \ll \frac{x}{\log x}\cdot L\big(c(\log x/(\log\log x)^2)^{1/(6d)}\big)^{-1}.
\]
Using our bound $L(Q)\gg Q/\log Q$, we obtain
\[
|\calV(x)| \ll \frac{x}{(\log x)^{1+1/(6d)}} (\log\log x)^{1+1/(3d)}.
\]
Now suppose that GRH holds.  Theorem~3.3(ii) of \cite{large-sieve} implies that
\[
|\calV(x)| \ll \big(\frac{x}{\log x} + \max_{D'\in\mathcal{Z}(Q)} |H_{D'}| \cdot \sum_{D\in\mathcal{Z}(Q)} |H_D^\sharp| |H_D|\cdot x^{1/2}\log x \big) L(Q)^{-1}.
\]
Using (\ref{E:H bounds}), $L(Q)\gg Q/\log Q$ and $|\mathcal{Z}(Q)|\leq Q$, we obtain the bound
\[
|\calV(x)| \ll \Big(\frac{x}{\log x} + Q^d \cdot |\mathcal{Z}(Q)| Q^r Q^d x^{1/2}\log x \Big) (\log Q)/Q \leq \Big(\frac{x}{\log x} + Q^{f}x^{1/2}\log x \Big) (\log Q)/Q
\]
where we have set $f:=2d+r+1$.  Setting $Q$ equal to $(x^{1/2}/(\log x)^2)^{1/f}$, we deduce that
\[
|\calV(x)|\ll \frac{x}{\log x} (\log Q)/Q \ll x/Q = x^{1-\frac{1}{2f}} (\log x)^{2/f}.
\]
Note that Theorem~3.3 of \cite{large-sieve} required our assumption that the representations $\{\rho_{A,\ell}\}_\ell$ are independent.
\end{proof}

We finally begin the proof of Theorem~\ref{T:main effective}.   Fix a maximal torus $\T$ of $\G_A$.   In \S\ref{SS:Galois action}, we defined a homomorphism $\psi_v\colon \Gal_\QQ \to \Aut(X(\T))$ for every place $v\in \calS_A$.   The following is an effective version of Proposition~\ref{P:sieving}.

\begin{prop}  \label{P:sieving effective}
Fix a finite extension $L$ of $k_{\G_A}$.   Let $\scrS$ be the set of places $v\in \calS_A$ for which $\psi_v(\Gal_L) \neq W(\G_A,\T)$.  Then
\[
|\scrS(x)| \ll \frac{x(\log\log x)^{1+1/(3d)}}{(\log x)^{1+1/(6d)}}.
\]
\item If GRH holds, then
\[
|\scrS(x)| \ll x^{1-\frac{1}{4d+2r+2}} (\log x)^{\frac{2}{2d+r+1}}.
\]
\end{prop}
\begin{proof}
The proof is the same as that of Proposition~\ref{P:sieving} with a few extra remarks.  After replacing $L$ by a finite extension, we may assume that $\T_L$ is split.  In the proof of Proposition~\ref{P:sieving} we chose a certain set of primes $\Lambda$ with positive density such that $\calG_{A,\ell}/\ZZ_\ell$ is a split reductive group scheme for all $\ell\in \Lambda$.  For a fixed non-empty subset $C$ of $W(\G_A,\T)$ which is stable under conjugation by $\Gamma$, we defined the set $\calV$ consisting of those places $v\in \calS_A$ for which $\bbar\rho_{A,\ell}(\Frob_v)\not\subseteq U_\ell$ for all $\ell\in \Lambda$ that satisfy $v\nmid \ell$; where the sets $U_\ell$ of $\bbar\rho_{A,\ell}(\Gal_K)$ are stable under conjugation and satisfy $|U_\ell|/|\bbar\rho_{A,\ell}(\Gal_K)|=|C|/|W(\G_A,\T)|+O(1/\ell)$.  We can then apply Proposition~\ref{P:large sieve} to bound $|\calV(x)|$.   The proposition then follows in the same manner as before.
\end{proof}

The proof of Theorem~\ref{T:main effective} is now identical to \S\ref{S:proof main} where we make use of Lemma~\ref{L:calSA effective} and Proposition~\ref{P:sieving effective} instead of using that certain sets have density 0.



\end{document}